\documentclass[12pt]{amsart}
\usepackage{amssymb,latexsym}
\usepackage{amsrefs}
\usepackage{eucal}
\usepackage{tikz}
\usetikzlibrary{arrows,intersections}

\usetikzlibrary{patterns}
\usepackage{times}

\swapnumbers

\theoremstyle{plain}
\newtheorem{blank}{}[section]

\newtheorem{proposition}[blank]{Proposition}
\newtheorem{corollary}[blank]{Corollary}

\newtheorem*{thma}{Theorem A}
\newtheorem*{thmaprime}{Theorem A\textprime}
\newtheorem*{thmb}{Theorem B}

\theoremstyle{definition}
\newtheorem{dblank}[blank]{}

\DeclareMathOperator{\diam}{diam}

\DeclareMathOperator{\card}{card}
\DeclareMathOperator{\Cl}{cl}
\DeclareMathOperator{\dis}{dist}
\DeclareMathOperator{\fsig}{{F}_{\sigma}}

\DeclareMathOperator{\gdelta}{{G}_{\delta}}

\DeclareMathOperator{\Int}{int}

\DeclareMathOperator{\dimcl}{dimcl}

\DeclareMathOperator{\Dim}{Dim}

\newcommand{\etal}{\textit{et al.}}
\newcommand{\ie}{\textit{i.e.}}
\newcommand{\egg}{\textit{e.g.}}
\newcommand{\cf}{\textit{cf.}}

\newcommand{\tx}{\textup}

\newcommand{\abs}[1]{\left\lvert#1\right\rvert}

\newcommand{\set}[1]{\{\, #1\,\}}
\newcommand{\Set}[1]{\left\{\, #1\,\right\}}

\newcommand{\inv}{^{-1}}
\newcommand{\rest}{{\restriction}}
\newcommand{\e}{\epsilon}
\newcommand{\N}{\mathbb N}
\newcommand{\Q}{\mathbb Q}
\newcommand{\R}{\mathbb R}
\newcommand{\Z}{\mathbb Z}
\newcommand{\rbar}{\overline\R}

\newcommand{\rr}{\mathfrak R}

\newcommand{\nbd}{\nobreakdash-\hspace{0pt}}

\newcommand{\Dsig}{\mathrm{D}_{\Sigma}}

\newcommand{\udim}{\overline{\dim}_{\mathrm{M}}}

\newcommand{\I}{\mathbb I}

\newcommand{\hdim}{\dim_{\mathrm{H}}}

\newcommand{\plc}{\,\underline{\phantom{x}}\,}

\DeclareMathOperator{\adim}{Dim}
\DeclareMathOperator{\adimcl}{Dimcl}
\DeclareMathOperator{\tdim}{ind}
\DeclareMathOperator{\tdimcl}{indcl}

\newcommand{\piperp}{\pi^{\perp}}

\DeclareMathOperator{\net}{net}

\makeatletter
   \def\th@plain{\slshape}
   \makeatother

\providecommand{\nopunct}{\setlength{\spacefactor}{1001}}

\begin{document}

\title
{Metric dimensions and tameness in expansions of the real field}
\author
[P. Hieronymi]
{Philipp Hieronymi}
\address
{Department of Mathematics\\University of Illinois at Urbana-Champaign\\1409 West Green Street\\Urbana, IL 61801}
\email{p@hieronymi.de}
\urladdr{http://www.math.uiuc.edu/\textasciitilde phierony}

\author[C. Miller]{Chris Miller}
\address
{Department of Mathematics\\
The Ohio State University\\
231 West~18th Avenue\\
Columbus, Ohio 43210, USA}
\email{miller@math.osu.edu}
\urladdr{https://people.math.osu.edu/miller.1987}

\begin{abstract}
For first-order expansions of the field of real numbers, nondefinability of the set of natural numbers is equivalent to equality of topological and Assouad dimension on images of closed definable sets under definable continuous maps.
\end{abstract}


\thanks{
\today.
This document should be regarded as preliminary;
some version of it has been submitted for publication.
Comments are welcome.
Miller is the corresponding author.}

\thanks{This work partially supported by:
NSF grant DMS-1300402 and UIUC Campus Research Board awards 13086 and 14194 (Hieronymi);
NSF grant DMS-1001176 (Miller).}



\maketitle

\section{Introduction}

We investigate relations between metric dimensions in real euclidean spaces and definability in expansions of $\rbar:=(\R,+,\cdot,(r)_{r\in\R})$,
the real field with constants for all real numbers.
The intended audience includes both mathematicians and logicians.
The main result is a dichotomy:
Roughly, any given expansion of $\rbar$ either defines all real projective sets or there is a striking agreement of various dimensions on an important subcollection of the definable sets.

For readers familiar with basic first-order logic (FOL, for short), structures in a given language, and expansions thereof, are defined as usual (see, \egg, Marker~\cite{marker}*{1.4.15}).
As the language of $\rbar$ contains a constant for each real number, there is no issue of whether ``definable'' means ``with parameters'' or ``without parameters'' in expansions of $\rbar$.
We identify interdefinable structures.

We ultimately refer readers who are not familiar with FOL to van den Dries and Miller~\cite{geocat} for an essentially FOL-free introduction to definability theory over $\rbar$.
However, we do give here some needed definitions and conventions.
First,
a \textbf{structure} on a set $X$ is
a sequence $\mathfrak S:=(\mathfrak S_m)_{m=1}^\infty$
such that for each $m$:
\begin{enumerate}
\item
$\mathfrak S_m$ is a boolean algebra of subsets of $X^m$ (the $m$\nbd fold cartesian power of $X$).
\item
If $S\in \mathfrak S_m$, then $\{S\times X, X\times S\}\subseteq \mathfrak S_{m+1}$.
\item
If $S\in \mathfrak S_{m+1}$, then the projection of $S$ on the first $m$ coordinates is in $\mathfrak S_m$.
\item
If $1\leq i\leq j\leq m$, then $\set{x\in X: x_i=x_j}\in \mathfrak S_m$.
\end{enumerate}
A \textbf{structure on} $\rbar$ is a structure on $\R$ such that (the graphs of) addition and multiplication belong to $\mathfrak S_3$ and every real singleton $\{r\}$ belongs to $\mathfrak S_1$.
(In the terminology of~\cite{geocat}, a structure on $\rbar$ is a structure on $(\R,+,\cdot\,)$
such that $\mathfrak S_1$ contains every real singleton.)
In this paper, we deal almost exclusively with structures on $\rbar$, but it is often useful to know there is a broader context.

We ask readers not familiar with FOL to take on faith that for every expansion $\rr$ of $\rbar$ in the sense of FOL there is a structure $\mathfrak S$ \tx($:=\mathfrak S(\rr)$\tx) on $\rbar$ such that each $\mathfrak S_m$ consists of the subsets of $\R^m$ that are definable in $\rr$, and
for every structure $\mathfrak S$ on $\rbar$ there is an expansion $\rr$ of $\rbar$ in the sense of FOL such that each $\mathfrak S_m$ consists of the subsets of $\R^m$ that are definable in $\rr$.
Thus, for the purposes of this paper, readers can take ``expansion of $\rbar$'' to mean ``structure on $\rbar$'' and definability to mean belonging to the appropriate $\mathfrak S_m$.
(It is crucial to understand that the notion of definability must always be taken with respect to some given structure.)
We regard maps as set-theoretic objects, so to say that a map $\R^m\supseteq X\to \R^n$ is definable is to say that its graph is definable.

Expansions of $\rbar$ can be partially ordered by setting $\rr\leq \rr'$ if every set definable in $\rr$ is definable in $\rr'$.
There is a maximal element:
Take each $\mathfrak S_m$ to be the power set of $\R^m$.
By definition, $\rbar$ is a minimal element.
The subsets of $\R^m$ definable in $\rbar$ consist of the \textbf{semialgebraic }sets in $\R^m$, that is, finite unions of sets of the form
$$
\set{x\in \R^m: f(x)=0, g_1(x)>0,\dots,g_j(x)>0}
$$
where $f,g_1,\dots,g_j\in\R[x_1,\dots,x_m]$; \cf~\cite{geocat}*{2.5.3}.
We say that $\rr$ and $\rr'$ are \textbf{interdefinable} if $\rr\leq\rr'$ and $\rr'\leq\rr$, equivalently, if $\mathfrak S(\rr)=\mathfrak S(\rr')$.

For any expansion $\rr$ of $\rbar$ and sequence $\mathcal A:=(\mathcal A_m)_{m=1}^\infty$ of collections $\mathcal A_m$ of subsets of $\R^m$ there is a minimal $\rr'\geq \rr$ such that, for every $m$, every $A\in \mathcal A_m$ is definable in~$\rr'$.
Rather than declare notation for this level of generality, we shall use various ways to indicate expansions of $\rr$ by sets of interest.
For example, if $E\subseteq \R^n$, then we let $(\rr,E)$ denote a minimal $\rr'$ that defines $E$ as well as every set definable in $\rr$.

The reader should now be able to apply~\cite{geocat} here as needed (begin with its second section and first two appendices).

What can be said about the lattice (up to interdefinability) of all expansions of $\rbar$?
For reasons that we shall not attempt to explain here, we regard this question as too vague at best and intractable at worst.
However, there are more reasonable versions, as we now begin to explain.

\subsection*{The real projective hierarchy}
Consider $(\rbar,\N)$, the expansion of $\rbar$ by the set of nonnegative integers~$\N$.
Logicians will immediately observe that $\operatorname{Th}(\R,+,\cdot,\N)$ is undecidable.
But much more is true:
$(\rbar,\N)$ defines every real Borel set (see, \egg, Kechris~\cite{kechris}*{(37.6)}), hence also every real projective set in the sense of descriptive theory (\cite{kechris}*{Chapter~V}).
Thus, the definable sets of $(\rbar,\N)$ comprise the real projective hierarchy.
To put this another way, the study of the definable sets of $(\rbar,\N)$ is essentially classical descriptive set theory.
Thus, even set-theoretic independence can arise from seemingly-innocent questions such as whether every set definable in $(\rbar,\N)$ is Lebesgue measurable (Solovay~\cite{solovay}).
Because of such complications, nondefinability of $\N$ is generally regarded as necessary for the definability theory of any particular expansion of $\rbar$ to be analyzable by model-theoretic methods.

The question arises: What can be said about expansions of $\rbar$ that do not define $\N$?
This question underlies the ``tameness program over $\rbar$''; see the introduction of~\cite{tameness} for a more detailed exposition.
The main result of this paper (Theorem~A, below) is one answer, but we need a few more preliminaries for its statement.

Unless indicated otherwise, the variables $j,k,l,m,n,p$ range over $\N$.
Given a set $X$, its power set is denoted by $\mathcal P(X)$ and its $n$\nbd fold cartesian
power by $X^n$, with $X^0$ the one-point set $\{0\}$.
For convenience, we often
identify $X^m\times X^n$ with $X^{m+n}$, and $(X^m)^n$ with $X^{mn}$; in particular,
$X^m\times X^0\cong X^0\times X^m\cong X^{m}$.
By a \textbf{coordinate projection} of $S\subseteq X^n$ we mean the image of $S$ under a map
\begin{equation*}
(x_1,\dots,x_n)\mapsto
\left(x_{\lambda(1)},\dots,x_{\lambda(m)}\right)\colon X^n\to X^m
\end{equation*}
where $0\leq m\leq n$ and $\lambda\colon \{1,\dots,m\}\to \{1,\dots,n\}$ is strictly
increasing.
We usually say just ``projection'' instead of ``coordinate projection''.

If $X$ is regarded as a subset of some topological space,
then we denote its interior by
$\Int X$ and closure by $\Cl X$.
We say that $X$: \textbf{has interior} if $\Int X\neq \emptyset$; \textbf{has no interior} if $\Int X=\emptyset$; and is \textbf{constructible} if it is a boolean combination of open sets.
If $X$ is given a topology, then $X^n$ is always regarded in the product topology.

Given a topological space $X$, we define a sequence of functions
$$
\dim_X:=\bigl (\dim_{X,m}\colon \mathcal P(X^m)\to\N\cup\{-\infty\}\bigr )_{m\in\N}
$$
by letting $\dim_{X,m}A$ (for $A\subseteq X^m$) be the supremum of all $k\in\N$ such that some projection of $A$ on $X^k$ has interior.
Observe that $\dim_{X,m}\emptyset=-\infty$ for all $m\in\N$.
We usually suppress the subscripts, writing just $\dim$ instead of $\dim_{X,m}$.
(Some ambiguity can then arise---\egg, if $X\neq \emptyset$, then
$\dim_{X,2} X^2=2$ and $\dim_{X^2,1}X^2=1$---but intent should always be clear from context.)
We also tend to shorten ``$\dim \Cl$'' to ``$\dimcl$''.
Our use of ``$\dim$'' here conflicts with some of the dimension theory literature
but is consistent with some key model-theoretic sources that we cite ($\dim$ is often useful when dealing with structures on~$X$).
Some properties (proofs are exercises):
\begin{enumerate}
\item
If $A\subseteq X^m$, then $\dim A=m$ if and only if $A$ has interior.
\item
$\dim$ is \textbf{increasing}\tx; \ie\tx, $\dim A\leq \dim B$ for all $m\in\N$ and $A\subseteq B\subseteq X^m$.
\item
$\dim$ is \textbf{stable} on compact sets\tx; \ie,
$\dim (A\cup B)=\max (\dim A,\dim B)$
for all $m\in\N$ and compact $A,B\subseteq X^m$.
\item
$\dim$ is \textbf{logarithmic}, that is, $\dim(A\times B)=\dim A+\dim B$ for all $A\subseteq X^m$ and $B\subseteq X^n$.
\item
If $A\subseteq X$, then $\dim A=0$ if and only if $A$ is nonempty and has no interior.
\item
$\dim$ is stable on constructible subsets of $X$.
\item
If $X$ is a metric space\tx, then $\dim \leq \hdim$, where $\hdim$ denotes Hausdorff dimension.
\end{enumerate}
(Hint for (6): constructible sets either have interior or are nowhere dense.
Hint for (7): $\hdim$ cannot increase under Lipschitz maps.)

\subsection*{Some global assumptions:\nopunct}
\begin{itemize}
\item
we regard $\R^n$ in the usual box topology, but we also employ the usual notation and conventions for working with the extended real numbers $\R\cup\{\pm\infty\}$;
\item
$E$ denotes an arbitrary subset of an arbitrary $\R^n$ unless otherwise indicated (\egg, by explicitly writing ``$n=1$'' or ``$E\subseteq \R$'');
\item
$\rr$ denotes a fixed, but arbitrary, expansion of $\rbar$;
\item
``definable'' means ``definable in $\rr$'' unless indicated otherwise (possibly only by context).
\end{itemize}

Many tameness conditions on $\rr$ imply that every definable subset of $\R$ either has interior or is nowhere dense, a condition that yields many desirable consequences for the definable sets; see~\cite{tameness}.
It would take us too far afield to go into details here, but we do need one basic result now:

\begin{blank}[see~\cite{tameness}*{\S7, Main Lemma}] \label{basicnwd}
The following are equivalent.
\begin{itemize}
\item
Every definable subset of~$\R$ has interior or is nowhere dense.
\item
$\dim=\dimcl$ on all definable subsets of~$\R$.
\item
$\dim=\dimcl$ on all definable sets.
\item
Every definable set either has interior or is nowhere dense.
\end{itemize}
\end{blank}
Thus, the technical dimensional condition  ``$\dim=\dimcl$ on all definable sets'' corresponds to a natural tameness condition on $\rr$,
namely, that each definable set either have interior or be nowhere dense.
Nevertheless, whatever good properties $\dim$ or $\dimcl$ might have, they cannot be regarded as interesting \emph{topological} dimensions because $\dim X=1$ for each fixed nonempty topological space~$X$.
This leads naturally to wondering what can be said about $\rr$ under assumptions on various topological or metric dimensions.

One classical dimension for topological spaces $X$ is the \textbf{small inductive dimension}, $\tdim X$, defined inductively by: $\tdim \emptyset= -\infty$; for $X\neq \emptyset$, $\tdim X$ is
the infimum of all $k$ such that, for every $x \in X$ and open neighborhood $V$ of $x$, there is an open $U$ such $x\in \Cl U\subseteq V$ and the boundary of $U$ (regarded as a topological space via the subspace topology) has $\tdim<k$.
(Often, $\tdim\emptyset$ is defined to be $-1$, but we prefer $-\infty$ for technical reasons.)
The definition is then extended to subsets of $X$ by passing to the subspace topology.
As with $\dim$, we tend to write $\tdimcl$ instead of $\tdim\Cl$.
Although defined for every topological space, $\tdim$ is not particularly well behaved on all topological spaces, nor in this generality does it bear any fixed relation (always $\leq $, always $=$, or always $\geq$) to certain other reasonable notions of topological dimension.
But heuristically, all reasonable notions of topological dimension
coincide on separable metrizable spaces (for more precision, see Engelking~\cite{engelking}, our primary reference for topological dimension theory).
Thus, when working with subsets of $\R^n$, we are justified in regarding $\tdim$ as \emph{the} topological dimension.

The relation between $\tdim$ and $\dim$ can be subtle, even in real $n$\nbd spaces.
While it is easy to see that $\tdim=\dim$ on all subsets of $\R$, we have only $\tdim\leq \dim$ on all subsets of $\R^2$ (with inequality possible, even on compact sets), and each of $\dim <\tdim $, $\dim =\tdim $ and $\dim >\tdim$ occur even among the $\gdelta$ subsets of $\R^3$.
On the other hand, $\tdim\leq \dim$ on all $\fsig$ sets; in particular, $\tdimcl\leq\dimcl$ always.
(See~\ref{dimtdim} for more details.)
It is thus natural to wonder what can be said about $\rr$ if either  $\tdim=\dim$ on all definable sets or $\tdim=\tdimcl$ on all definable sets (\cf~\ref{basicnwd}).
We shall indeed produce some answers here,
but because we are thinking about $\R$ as a metric space, we prefer to bypass this level of generality in favor of ``dimensional coincidences'' of $\dim$ or $\tdim$ with various metric dimensions.
We do have a particular metric dimension in mind, called \textbf{Assouad dimension} by some, but we postpone (to Section~\ref{S:dims}) giving the definition as we think its technical nature would only be distracting at this point; for now, we only list enough of its basic properties over real $n$\nbd spaces to justify our interest.

\begin{dblank}\label{adimfactsintro}
There is a sequence
$
\displaystyle\adim:=\bigl(\adim_m\colon \mathcal P(\R^m)\setminus\{\emptyset\}\to [0,\infty)\bigr)_{m\in\N}
$
such that for all $m\in\N$:
\begin{enumerate}
\item
$\adim_m =m$ on nonempty open subsets of $\R^m$.
\item
$\adim_m$ is increasing.
\item
$\adim_m$ is stable on closed sets.
\item
$\adim_p f(A)=\adim_m A$ for all $A\subseteq \R^m$, $p\in\N$ and bi-Lipschitz maps
$f\colon A\to \R^p$ (\textbf{Lipschitz invariance}).
\item
$\adim_m=\adim_m\Cl$ (\textbf{closure invariance}).
\item
$ \adim_m\geq \udim$ on bounded sets, where $\udim$ denotes upper Minkowski dimension.
\item
$\adim_1 \set{1/k: 0\neq k\in\N}=1$.
\end{enumerate}
As with $\dim$, we tend to suppress the subscripted~$m$.
We shall also have $\adim \emptyset=-\infty$.
\end{dblank}

The reader need not yet know the definition of $\udim$ (also given in Section~\ref{S:dims}); its appearance here is only to lend credence to an empirical heuristic: \emph{If $E\neq \emptyset$ and $\tdim E=\adim E$, then all dimensions commonly encountered in geometric measure theory, fractal geometry and analysis on metric spaces are equal on $E$.}\footnote{Conventions for dealing with dimensions of $\emptyset$ vary in the literature.}
We refer the reader to Luukkainen~\cite{luu} for
history and explanations of technical significance (but there, $\tdim$ is denoted by $\dim$,  $\adim$ by $\dim_{\mathrm A}$, and  $\udim$ by $\overline{\dim}_{\mathrm B}$).
Further information on metric dimension theory can be found in Falconer~\cite{falconerbook}, Mattila~\cite{mattila} and Robinson~\cite{MR2767108} (but again, notation and conventions vary).

\subsubsection*{Remarks}
(a)~We might reasonably demand that the first four properties of~\ref{adimfactsintro} hold in order for a sequence of functions
$
\bigl(\mathcal P(\R^m)\setminus\{\emptyset\}\to [0,\infty)\bigr)_{m\in\N}
$
to be considered as a system of metric dimensions for real euclidean spaces, but preservation under closure (which renders superfluous the phrase ``on closed sets'' in the third property) fails for $\hdim$.
(b)~It is known that $\udim$ satisfies the first five properties on bounded sets, but $\udim\set{1/k: 0\neq k\in\N}=1/2$.
(c)~If $E\neq \emptyset$ and $\dim E=\adim E$, then heuristically, all
\emph{metric} dimensions should agree on $E$. (It is a routine consequence of $\udim\leq \adim$ on bounded sets that $\hdim\leq \adim$, and
recall that $\dim\leq \hdim$.)

We are now ready to state our main result:

\begin{thma}
If $\rr$ does not define $\N$\tx, then $\tdim=\adim=\dim$ on projections of closed definable sets.
\end{thma}

An equivalent local version:

\begin{thmaprime}
If $E$ is closed\tx, $f\colon E\to \R^p$ is continuous\tx, and $(\rbar,f)$ does not define $\N$\tx, then $\tdim f(E)=\adim f(E)=\dim f(E)$.
\end{thmaprime}

(By basic definability, the conclusion of
Theorem~A is equivalent to $\tdim=\adim=\dim$ on images of closed definable sets under definable continuous maps.)
Heuristically:
\emph{In order for the definability theory of $\rr$ to be amenable to model-theoretic methods, we must have $\tdim=\adim=\dim$ at least on all images of closed definable sets under definable continuous maps.}

It is easy to see that the converse of Theorem~A
holds, indeed,
if either $\adim=0$ on all countable compact definable subsets of $\R$
or $\tdim\geq \dim$ on all compact definable subsets of $\R^2$, then $\rr$ does not define $\N$.
For the former, recall~\ref{adimfactsintro}.5 and~\ref{adimfactsintro}.7; for the latter, see~\ref{rotatecantor} below.
We now collect a few easy applications of Theorem~A to the tameness program.

\begin{corollary}\label{Efractal}
If $\tdim E\neq \adim E$\tx, then
either $\tdim E\neq \tdimcl E$ or $(\rbar,E)$ defines $\N$.
\end{corollary}

\begin{proof}
If $(\rbar,E)$ does not define $\N$, then $\tdim\Cl E=\adim\Cl E$ by Theorem~A. Recall that $\adim=\adimcl$.
\end{proof}

Heuristically: \emph{If any of the commonly-encountered metric dimensions fail to coincide on $E$, it is only because $E$ is \emph{topologically} noisy or it encodes enough information to define \textup(over $\rbar$\textup) all real projective sets.}

We can add another condition to~\ref{basicnwd}:

\begin{corollary}\label{dimadimnwd}
Every definable subset of $\R$ has interior or is nowhere dense if and only if
$\dim=\adim$ on all definable sets.
\end{corollary}

\begin{proof}
Suppose that every definable subset of $\R$ has interior or is nowhere dense.
Then $\rr$ does not define $\Q$, hence neither $\N$, so
$\dimcl=\adimcl=\adim$ on all definable sets.
By~\ref{basicnwd}, we also have $\dimcl=\dim$ on all definable sets.
Hence, $\dim=\adim$ on all definable sets.

The converse is immediate from $\adimcl=\adim$.
\end{proof}

We can now address some issues raised earlier:

\begin{corollary}
\label{tdimtdimclthm}
The following are equivalent.
\begin{enumerate}
\item
$\tdim=\adim$ on all definable sets.
\item
$\tdim=\tdimcl$ on all definable sets.
\item
$\tdim=\dim=\adim$ on all definable sets.
\item
$\tdim=\dim$ on all definable sets.
\item
$\tdim\geq \dim$ on all definable sets.
\end{enumerate}
\end{corollary}

\begin{proof}
(1)$\Rightarrow$(2) is immediate from $\adimcl=\adim$.

(2)$\Rightarrow$(3).
If $\tdim=\tdimcl$ on all definable sets, then $\rr$ does not define~$\Q$, hence neither~$\N$, so
$\tdimcl=\adimcl=\dimcl$ on all definable sets by Theorem~A.
Thus,
$\tdim=\adim=\dimcl$ on all definable sets.
Since $\tdim=\dim$ on all subsets of $\R$, we also have $\dim=\dimcl$ on all definable sets by~\ref{basicnwd}, and so $\tdim=\adim=\dim$ on all definable sets.

(5)$\Rightarrow$(1).
Assume that $\tdim\geq \dim$ on all definable sets.
As mentioned earlier, $\tdimcl\leq \dimcl$ always,
so $\dim\leq \tdim\leq \dimcl$ on all definable sets.
Hence, it suffices by~\ref{basicnwd} and~\ref{dimadimnwd}
to show that every definable subset of $\R$ has interior or is nowhere dense; if
not, there would be a definable subset of $\R$ whose characteristic function has $\tdim=0$ and $\dim=1$.
\end{proof}

We do not know whether the above are also equivalent to
$\dimcl=\dim$ on all definable sets,
but we do have a partial answer:

%
%

\begin{corollary}\label{nocantorintro}
If $\rr$ defines no Cantor subsets of $\R$ and $\dim=\dimcl$ on all definable subsets of $\R$\tx,
then $\tdim=\adim=\dim$ on all definable sets.
\end{corollary}

(By ``Cantor'' we mean nonempty, compact, totally disconnected, and no isolated points.)

\begin{proof}
It is an exercise that the assumption is equivalent to: Every nonempty definable subset of $\R$ has interior or an isolated point.
Suppose that $E$ is definable.
By~\cite{tameness}*{3.4}, after permutation of coordinates there is a box $U\subseteq \R^{\dim E}$ and a continuous map $U\to \R^{n-\dim E}$ whose graph is contained in $E$.
Hence, $\tdim E\geq \dim E$.
As $E$ is arbitrary, we have $\tdim=\adim=\dim$ on all definable sets by~\ref{tdimtdimclthm}.
\end{proof}

There are expansions of $\rbar$ that define Cantor sets and every definable subset of $\R$ has interior or is nowhere dense; see Friedman \etal~\cite{fkms}, and~\cite{tamecantor}.
We do not know whether any of them satisfy $\dim\leq \tdim$ on all definable sets.

It is immediate from Theorem~A that if every definable set is a projection of a closed definable set, then
$\tdim=\adim=\dim$ on all definable sets.
We can do better:

\begin{corollary}\label{fsigintro}
If every definable subset of $\R$ is $\fsig$\tx, then $\tdim=\adim=\dim$ on all definable sets.
\end{corollary}

\begin{proof}
By~\ref{nocantorintro}, it suffices to show that then every definable subset of $\R$ has interior or is nowhere dense, and no Cantor subsets of $\R$ are definable.

Let $A\subseteq \R$ be definable and dense in a nonempty open interval $I$.
Since both $I\cap A$ and $I\setminus A$ are definable, they are each $\fsig$ by assumption.
As $I\setminus A$ has no interior, it is meager (in the sense of Baire).
Hence, $I\cap A$ is nonmeager, and thus has interior.

If $A\subseteq \R$ is a Cantor set, then the complement in $A$ of the set of left endpoints of the complementary intervals of $A$ is definable (exercise) and not $\fsig$ (because it is both comeager and codense in the closed set $A$).
\end{proof}

In particular,

\begin{corollary}\label{intorcntbl}
If every definable subset of $\R$ has interior or is countable\tx, then $\tdim=\adim=\dim$ on all definable sets.
\end{corollary}

There are many examples of $\rr$ such that every definable subset of~$\R$ has interior or is countable.
Indeed, there are many examples such that every definable subset of~$\R$ has interior or is finite (o\nbd minimality)---see~\cite{geocat} for just a few---but there are also examples that define countably infinite sets; for some,
see van den Dries~\cite{twoz}, Friedman and Miller~\cites{ominsparse,fast}, Miller and Tyne~\cite{itseq}, \cite{tameness}, and~\ref{infrank} below.




Two important precursors to this paper should be mentioned even though there is no formal dependence.
(i)~In joint work with Fornasiero~\cite{fhm}, we established that if $E\subseteq \R$ is nonempty, bounded and nowhere dense\tx, and $(\rbar, E)$ does not define $\N$, then $\udim E=0$.
Ideas from the proof\footnote{One of which, as noted in~\cite{fhm}, was communicated to Miller by K.~Falconer in personal correspondence. It was also Falconer who first drew our attention to Assouad dimension.}
play a crucial role in the proof of Theorem~A.
(ii)~In 2011, Fornasiero announced
a weaker version of
Theorem~A, namely,
that it holds with $\hdim$ in place of $\adim$ (subject to the usual proviso that $\hdim\emptyset$ is often defined to be $0$;
see Fornasiero, Hieronymi and Walberg~\cite{fhw} for more current information.
Thus, it would suffice for us here to show just that $\hdim=\adim$ on all projections of nonempty closed definable sets
if $\rr$ does not define $\N$.
But we have not been able to accomplish this, as we do not know of any appropriately useful criteria for checking $\hdim=\adim$.
Neither does Fornasiero's proof generalize to $\adim$ as it relies on properties of $\hdim$ that are known to fail for $\adim$.
On the other hand, we could appeal to~\cite{fhw}
for the ``$\tdim=\dim$'' part of
Theorem~A,
but we shall give our own proof (\ref{uniform}) as a natural step toward establishing
Theorem~A.

Here is an outline of the rest of this paper.
Section~\ref{S:prelims} consists of some technical preliminaries, mostly topological, culminating in the proof of the ``$\tdim=\dim$'' part of
Theorem~A.
While the proof of the ``$\dim=\adim$'' part of
Theorem~A requires an actual definition of $\adim$, the cases $n=1$ and $\dim=0$ require only two easily-stated properties beyond those of~\ref{adimfactsintro}, and the proofs illustrate some key ideas needed for the proof of
Theorem~A as a whole.
Hence, with an eye toward possible generalization, we present in Section~\ref{S:fhmgen} an axiomatic proof of these basic cases.
We also dispose of the trivial (relative to existing technology) case that $\rr$ defines no infinite discrete closed subsets of $\R$.
In Section~\ref{S:dims}, we give the definitions of $\adim$ and $\udim$, and prove Theorem~A via a more technical result (Theorem~B).
We conclude in Section~\ref{S:conc} with a few remarks.

%
%
%
%
%

\section{Preliminaries}\label{S:prelims}

Let $X,Y$ be sets.
Given $f\colon X\to Y$
and $S\subseteq X$, we let $f\rest S$ denote the restriction of
$f$ to $S$.
We identify a function $f\colon X^0\to
Y$ with the constant $f(0)\in Y$.
Given $Z\subseteq X\times Y$ and $x\in X$, we
put $Z_x:=\set{y\in Y: (x,y)\in Z}$;
this notation will be used primarily in settings where both $X$ and $Y$ are finite cartesian powers of~$\R$, but it can also be used to resolve any potential ambiguity caused by the usual kinds of subscripting of indices that arise in mathematical exposition.
If $Y$ is a product $Y_1\times Y_2$ and $(x,u)\in X\times Y_1$, we tend to abbreviate $Z_{(x,u)}$ by $Z_{x,u}$; then
$(Z_x)_u=Z_{x,u}$.



If $(X,\mathrm{d})$ is a metric space, then $X^n$ is assumed to be equipped with the $\sup$ metric
$(a,b)\mapsto
\sup\{\,\mathrm{d}(a_i,b_i):i=1,\dots,n\}$.
In particular, we use the
$\sup$ norm $\abs{a}:=\sup\{\abs{a_i}: i=1,\dots,n\}$ in $\R^n$.



For
our purposes, \textbf{intervals} (in $\R$) always have interior.
The
usual notation is employed for the various kinds of intervals.
We
sometimes write $\R^{>0}$ or $(0,\infty)$ instead of $(0,+\infty)$.
The interval $[0,1]$ is denoted by $\I$.
A \textbf{box} in~$\R^n$ is
an $n$\nbd fold product of open intervals, and a \textbf{closed
box} is a product of closed intervals.


The set $E$ is
\textbf{discrete} if all of its points
are isolated, \textbf{locally closed} if it is open in its closure,
and \textbf{dense in $C\subseteq \R^n$} if $\Cl(C\cap E)=\Cl(C)$.


\begin{dblank}\label{tdimfacts}
Some basic properties of $\tdim$ in $\R^n$ (\cite{engelking}*{1.8}):
\begin{enumerate}
\item
$\tdim E=n$ if and only if $E$ is nonempty and open.
\item
$\tdim$ is increasing.
\item
$\tdim$ is \textbf{countably stable} on closed sets, that is, $
\tdim \bigcup_{k\in\N}F_k=\max_{k\in\N}\tdim F_k
$
for any sequence $(F_k)_{k\in\N}$ of closed subsets of $\R^n$.
(Hence,
if
$E$ is $\fsig$, then $\tdim E$ is realized on some compact subset of $E$.)
\item
$\tdim$ is \textbf{sublogarithmic}, that is, if
$A\subseteq \R^j$ and $B\subseteq\R^k$, then
$
\tdim (A\times B)\leq \tdim A + \tdim B.
$
\item
$\tdim f(E)=\tdim E$ for homeomorphisms $f\colon E\to f(E)\subseteq \R^p$ and $p\in\N$ (\textbf{topological invariance}).
\end{enumerate}
\end{dblank}


\begin{dblank}\label{dimRn}More on $\dim$ in $\R^n$ (proofs are exercises):
\begin{enumerate}
\item
$\dim$ is countably stable on closed sets (by the Baire Category Theorem).
\item
For $k\leq m\leq n$ and $*$ in $\{\leq,\geq,=,<,>\}$\tx, the set
$
\set{x\in \R^m: \dim E_x * k}
$
is definable.
\end{enumerate}
\end{dblank}


\subsubsection*{Note to logicians:} (2)~holds with ``$\emptyset$\nbd definable in $(\R,<,E)$'' instead of ``definable in $\rr$''.

\begin{dblank}\label{dimtdim}
Some relations between $\dim$ and $\tdim$ in $\R^n$ (proofs are exercises):
\begin{enumerate}
\item
$\tdim =\dim$ on subsets of $\R$.
\item
$\dim E=0\Rightarrow \tdim E = 0$.
\item
$\tdim\leq \dim$ on subsets of $\R^2$.
\item
$\tdim \leq \dim $ on compact sets~\cite{engelking}*{Exercise~1.8.C}.
\item
$\tdim \leq \dim $ on $\fsig$ sets.
\item
$\tdimcl\leq \dimcl$.
\end{enumerate}
\end{dblank}

(For (5), use~(4) and that $\tdim$ and $\dim$ are countably stable on closed sets.)

\subsubsection*{Remark}
There are $\gdelta$ sets in $\R^3$ such that $\tdim>\dim$;
see~\cite{engelking}*{1.10.23}.

\begin{dblank}\label{rotatecantor}
If $C$ is the standard ``middle-thirds'' Cantor set, then $\tdim\set{(x-y,x+y): x,y\in C}=0$
and $\dim\set{(x-y,x+y): x,y\in C}=1$ (recall that the difference set of $C$ has interior).
Hence, strict inequality is possible in~\ref{dimtdim}.3, even on compact sets.
This also shows that $\dim$ is not necessarily preserved under isometries.
(This elegant example was brought to our attention by A.~Fornasiero.)
\end{dblank}

\begin{dblank}\label{familydim}

We shall have need later for notions of uniformity for collections of sets.
Given a set $X$,
we say that a map
$f\colon \mathcal P(\mathcal P(X))\to\R\cup\{\pm\infty\}$ is
\textbf{increasing} if $f(\mathcal A)\leq f(\mathcal B)$ for all
$\mathcal A,\mathcal B\subseteq \mathcal P(X)$ such that
$\mathcal A$ is a refinement of $\mathcal B$, and \textbf{stable} if
$$f(\set{A\cup B: (A,B)\in\mathcal A\times \mathcal B})=\max(f(\mathcal A),f(\mathcal B))$$ for all $\mathcal A,\mathcal B\subseteq \mathcal P(X)$.
There are obvious modifications of ``logarithmic'' and ``sublogarithmic''
for cartesian products.
For illustrative purposes, if $X$ is a topological space, then we put
$\dim \mathcal A=\sup\set{\dim A:A\in\mathcal A}$ for $\mathcal A\subseteq \mathcal P(X^m)$, suppressing the subscript on $\dim_m$ as usual.
Some easy basic facts:
\begin{itemize}
\item
$\dim A=\dim\{A\}$ for all $A\subseteq X^m$.
\item
If $\mathcal A\subseteq \mathcal P(X^m)$, then $\dim \mathcal A=m$ if and only if some $A\in \mathcal A$ has interior.
\item
If $\mathcal A\subseteq \mathcal P(X)$, then $\dim \mathcal A=0$ if and only if some $A\in\mathcal A$ is nonempty and no $A\in\mathcal A$ has interior.
\item
$\dim$ is increasing.
\item
$\dim$ is stable on the collection of compact subsets of $X^m$.
\item
$\dim$ is logarithmic.
\end{itemize}
\end{dblank}

The \textbf{open core} of $\rr$, denoted by $\rr^\circ$, is the expansion of the set $\R$ by all open sets (of any arity) definable in $\rr$.
(Equivalently: $\rr^\circ$ is identified with the smallest structure $\mathfrak S$ on $\R$ such that each $\mathfrak S_m$ contains every open subset of $\R^m$ definable in $\rr$.)
Some easy observations:
\begin{blank}\label{opencoredsig}
\begin{enumerate}
\item
$\rr^\circ$ is interdefinable with the expansion of $\rbar$ by all projections of closed sets definable in $\rr$.
\item
$\rr$ defines $\N$ if and only if $\rr^\circ$ defines $\N$ if and only $\rr^\circ$ is interdefinable with $(\rbar,\N)$.
\end{enumerate}
\end{blank}

The structure $\rr$ is \textbf{o-minimal} (short for ``order minimal'') if every definable subset of $\R$ is a finite union of points and open intervals.
O\nbd minimal expansions of $\rbar$ are often regarded as the best-behaved expansions of $\rbar$, indeed, they have so many nice properties that it takes several pages just to state even the most important ones; see, \egg,~\cite{geocat}*{\S4}.
It is immediate from definitions that

\begin{blank}
$\rr$ is o\nbd minimal if and only if
$\rr^\circ$ is o\nbd minimal and every
subset of $\R$ definable in $\rr$ has interior or is nowhere dense.
\end{blank}

More substantial:

\begin{blank}\label{omincorereduction}
$\rr^\circ$ is o\nbd minimal if and only if $\rr$ defines no infinite discrete closed subsets of $\R$.
\end{blank}

\begin{proof}
The forward implication is immediate from definitions.
Suppose that every discrete closed subset of $\R$ definable in $\rr$ is finite.
By~\cite{hier2}*{Lemma~2}
every discrete subset of $\R$ definable in $\rr$ is finite.
By Miller and Speissegger~\cite{opencore}*{Theorem~(b)}, $\rr^\circ$ is o\nbd minimal.
\end{proof}

For some examples of non-o\nbd minimal expansions of $\rbar$ having o\nbd minimal open core,
see:
Dolich, Miller and Steinhorn~\cite{dms2}*{1.11 and~1.12} and~\cite{dms3};
G{\"u}nayd{\i}n and Hieronymi~\cite{GH};
and~\cite{opencore}*{4.2}.

\subsection*{Projections of closed definable sets}

Following~\cite{opencore}, we let $\Dsig(n)$ denote the collection of all subsets of $\R^n$ that are projections of closed definable sets.
If the ambient $\R^n$ is irrelevant or understood from context, then we usually write just $\Dsig$.
(Thus, we can now state the conclusion of Theorem~A
as ``$\tdim=\adim=\dim$ on $\Dsig$''.)
In this subsection, we review some basic known facts and establish a key technical result (\ref{uniform}) that disposes of the ``$\tdim=\dim$'' part of Theorem~A.

\subsubsection*{Remark}
As every $\Dsig$ set is $\fsig$,
one might be tempted to think of ``$\Dsig$'' as abbreviating ``definably $\fsig$'',
but caution is in order, as there at least two other reasonable notions for this:
(i)~definable and $\fsig$;
(ii)~definable and a countable union of closed definable sets.
In general, these notions differ from $\Dsig$, as well as from each other.

\begin{blank}
\begin{enumerate}
\item
$\Dsig$ is closed under projection\tx, cartesian product\tx, union and intersection.
\item
Every constructible definable set is $\Dsig$.
\item
If $Z\in \Dsig(m+n)$ and $x\in\R^m$\tx, then $Z_x\in \Dsig(n)$.
\end{enumerate}
\end{blank}

\begin{proof}
(1).
Closure under projection is immediate from definition,
and closure under cartesian product is nearly so.
We now deal with unions and intersections.
Let $E_1,E_2\in \Dsig(n)$.
For $i=1,2$, let $F_i\subseteq \R^{p_i}$ be closed and definable such that $E_i$ is a projection of $F_i$.
By permutation of coordinates, we may assume that $E_i$ is the projection of $F_i$ on the first $n$ coordinates.
After embedding into a common $\R^p$, we may take $p_1=p_2=p$.
Then $E_1\cup E_2$ is the projection on the first $n$~coordinates of $F_1\cup F_2$, and $E_1\cap E_2$ is the projection on the first $n$~coordinates of
\[
\set{(x,y,z)\in \R^{n+2p}: (x,y)\in F_1\And (x,z)\in F_2}.
\]

(2).
By Dougherty and Miller~\cite{DoM}, every constructible definable set is a boolean combination of open definable sets.
Hence, as $\Dsig$ contains all closed definable sets and is closed under unions and intersections, it suffices to show that all open definable sets are $\Dsig$.
If $E$ is open (or even just locally closed), then it is the projection on the first $n$ coordinates of
\[
\set{(x,t)\in E\times \R: t\dis(x,(\Cl E)\setminus E)\geq 1}.\qedhere
\]

(3).
$Z_x$ is the projection of $Z\cap (\{x\}\times \R^n)$ on the last $n$ coordinates.

\end{proof}

$\Dsig$ need not be closed under complementation; indeed, complements of $\Dsig$ sets need not be $\fsig$
(\egg, if $\rr$ defines $\N$, then $\Q$ is $\Dsig$ and $\R\setminus\Q$ is not $\fsig$).
If $\Dsig$ is closed under complementation, then $\Dsig=\mathfrak S(\rr^\circ)$.



We use the next result often.

\begin{blank}\label{Dsigequiv}
The following are equivalent.
\begin{enumerate}
\item
$E$ is $\Dsig$.
\item
There is a definable $X\subseteq \R^{>0}\times \R^n$ such that $E=\bigcup_{r>0}X_r$ and $(X_r)_{r>0}$ is an increasing family of compact sets.
\item
$E$ is the projection on the first $n$ coordinates of a closed definable $F\subseteq \R^{n+1}$.
\item
There are a closed definable $F\subseteq\R^{n+1}$ and a surjective continuous definable map $f\colon F\to E$.
\item
There are a closed definable $F\subseteq\R^p$ and a surjective continuous definable map $f\colon F\to E$.
\end{enumerate}
\end{blank}
\begin{proof}
For (1)$\Rightarrow$(2),
if $E$ is the projection $\pi F$ of a closed definable $F\subseteq \R^p$, then put $X=\set{(r,x): x\in \pi(F\cap [-r,r]^p)}$.
For (2)$\Rightarrow$(3),
let $F$ be the closure of $\set{(x,r)\in \R^{n+1}: x\in X_r}$.
The rest is routine.
\end{proof}


%

\begin{blank}\label{fiberdim}
If $E$ is $\Dsig$\tx, then
$\set{x\in \R^m: \dim E_x\geq k}$
is $\Dsig$.
\end{blank}

\begin{proof}
As $\Dsig$ is closed under finite unions and permutation of coordinates, it suffices to let $\pi$ be a projection $\R^{n-m}\to\R^k$ and show that $S:=\set{x\in \R^m: \text{$\pi(E_x)$ has interior}}$ is $\Dsig$.
Write $E=\bigcup_{r>0} X_r$ as in~\ref{Dsigequiv}; then
$$
S=\bigcup_{r>0}\set{x\in\R^m: \exists y\in\R^k\ \textstyle {\prod_{i=1}^k}[y_i,(1+1/r)y_i]\subseteq \pi(X_{r,x})}.
$$
Let $x\in\R^m$ and $r>0$.
As $X_r$ is compact, so is $\pi(X_{r,x})$, hence so is
$$
\set{x\in\R^m: \exists y\in\R^k\ \textstyle{\prod_{i=1}^k}[y_i,(1+1/r)y_i]\subseteq \pi(X_{r,x})}.
$$
The result follows.
\end{proof}

We say that a map $f$ is $\Dsig$ if it is $\Dsig$ as a set (that is, if its graph is $\Dsig$).

\begin{blank}
If $f\colon E\to\R^p$ is $\Dsig$\tx, then so are $E$\tx, $f(E)$\tx, and $f\inv(Y)$ for every $Y\in \Dsig(p)$.
\end{blank}

\begin{proof}
Observe that $E$ is the projection of $f$ on the first $n$ coordinates, $f(E)$ is the projection of $f$ on the last $p$ coordinates, and $f\inv(Y)$ is the projection on the first $n$ coordinates of $f\cap (\R^n\times Y)$.
\end{proof}

%

Given $f\colon  E\to \R^p$,
we let $\mathcal D(f)$ denote the set of all $x\in E$ such that $f$ is discontinuous at $x$.
For $s>0$, we let $\mathcal D(f,s)$ be the set of $x\in E$ such that the oscillation
of $f$ at $x$ (defined as is usual in analysis) is at least $s$.
Observe that $\mathcal D(f,s)$ is closed in $E$ and $\mathcal D(f)=\bigcup_{s>0}\mathcal D(f,s)$.
If $f$ is definable, then so is $\mathcal D(f,s)$.

\begin{blank}\label{D(f)}
If $E$ is $\Dsig$ and $f\colon  E\to \R^p$ is definable\tx, then $\mathcal D(f)$ is $\Dsig$.
\end{blank}

\begin{proof}
If $E=\bigcup_{r>0}X_r$ as in~\ref{Dsigequiv}, then
$\mathcal D(f)=\bigcup_{r>0}\mathcal D(f\rest X_r,1/r)$.
\end{proof}

%
%


%
%

Recall that a closed box in $\R^n$ is the closure of a nonempty open box in $\R^n$.
We define \textbf{rectangular partitions} of closed boxes by induction:
A rectangular partition of a closed interval $B\subseteq \R$ is a finite cover of $B$ by closed subintervals whose interiors are pairwise disjoint; a rectangular partition of a closed box $B\subseteq \R^{n+1}$ is the cartesian product of a rectangular partition of the projection of $B$ on the first $n$ variables with a rectangular partition of the projection of $B$ on the last variable.

Next is a key technical result.

\begin{proposition}\label{uniform}
The following are equivalent.
\begin{enumerate}
\item
$\rr$ does not define $\N$.
\item
For all definable $h\colon A^p\to\R$ such that $A\subseteq \R$ is discrete\tx, $h(A^p)$ is nowhere dense.
\item
If $E$ is $\Dsig$\tx, then for every $m\leq n$ and closed box $B\subseteq \R^{n-m}$ there is
a rectangular partition $\mathcal P$ of $B$ such that for each $u\in \R^m$ there exists $P\in \mathcal P$ disjoint from
$E_u\setminus\Int(E_u)$.
\item
Each $\Dsig$ either has interior or is nowhere dense.
\item
$\tdim=\dim$ on $\Dsig$.
\item
$\tdim\geq \dim$ on all compact definable subsets of $\R^2$.
\end{enumerate}
\end{proposition}

(The implication (1)$\Rightarrow$(3) can be regarded as a uniform version of (1)$\Rightarrow$(2), which is a theorem on its own; see~\cite{hier2}.)

\begin{proof}
(1)$\Rightarrow$(2) is a special case of~\cite{hier2}*{Theorem~A}.


(2)$\Rightarrow$(3).
By replacing $E$ with $\R\times E$, we may assume that $m>0$.
The result is trivial if $m=n$, and an easy induction allows us to reduce to the case that $n=m+1$, that is, each $E_u\subseteq \R$.

We first do the case that $E$ is bounded.
Since $\Dsig$ is closed under intersection and continuous definable maps,
we reduce to the case that $E\subseteq \I^{m+1}$ and $B=\I$
(recall that $\I=[0,1]$).
It suffices now to find $\e>0$ such that if $u\in \I^m$, then at least one of $E_u$ or
$\I\setminus E_u$ contains an interval of length $\e$.
This follows easily from cell decomposition (see, \egg, \cite{geocat}*{4.2}) if $\rr^\circ$ is o\nbd minimal, so assume that $\rr^\circ$ is not o\nbd minimal.
By~\ref{omincorereduction}, $\rr$ defines an infinite discrete closed $D\subseteq \R^{>0}$, and we have
$E=\bigcup_{r\in D} X_r$ where $(X_r)_{r\in D}$ is an increasing definable family of compact sets.
Define $Y,Z\subseteq D^2\times \I^m$ by $u\in Y_{r,s}$ if $X_{r,u}$ contains an interval of length $1/s$, and
$u\in Z_{r,s}$ if every interval contained in $\I\setminus X_{r,u}$ has length at most $1/s$.
It suffices to show that there is some $r\in D$ such that
$Z_{r,s}\setminus Y_{r,s}=\emptyset$ for every $s\in D\cap (r,\infty)$; suppose this fails.
Then we have a definable function $\alpha\colon D\to D$ given by
$$
\alpha(r)=\min\set{s\in D\cap (r,\infty): Z_{r,s}\setminus Y_{r,s}\neq \emptyset}.
$$
Observe that for every $r\in D$, both $Y_{r,\alpha(r)}$ and $Z_{r,\alpha(r)}$ are compact (for the former, recall the proof of~\ref{fiberdim}).
Let $r\in D$.
If $Z_{r,\alpha(r)}\cap Y_{r,\alpha(r)}=\emptyset$, let $\beta(r)$ be the lexicographic minimum of $Z_{r,\alpha(r)}$.
If $Z_{r,\alpha(r)}\cap Y_{r,\alpha(r)}\neq \emptyset$, let $\beta(r)$ be the lexicographic minimum of all $u\in Z_{r,\alpha(r)}$ such that
the distance of $u$ to $Y_{r,\alpha(r)}$ is maximal.
Then $\beta\colon D\to \R^m$ is definable and
$\beta(r)\in Z_{r,\alpha(r)}\setminus Y_{r,\alpha(r)}$
for all $r\in D$.
Since $X_r$ is closed, so is $X_{r,u}\setminus \Int(X_{r,u})$.
Define
$M\subseteq D\times \I^m\times \I$ by
$
(r,u,t)\in M
$
if and only $t$ is a midpoint of a complementary-in-$\I$ interval of $X_{r,u}\setminus \Int(X_{r,u})$.
Let $\sigma(r)$ be the successor of $r$ in $D$, that is, $\sigma(r)=\min(D\cap (r,\infty))$.
Now $M$ is definable, each $M_{r,u}$ is discrete, and the image of $M_{\alpha(r),\beta(r)}$ under the function $x\mapsto r+(\sigma(r)-r)x\colon \R\to\R$ is discrete and contained in the interval $(r,\sigma(r))$.
Let $A$ be the union of all such sets as $r$ ranges over $D$; then $A$ is discrete and definable, and for every $a\in A$ there are unique $r\in D$ and $t\in M_{\alpha(r),\beta(r)}$ such that $a=r+(\sigma(r)-r)t$.
Let
$g\colon M\to \R$ be given by
$(r,u,t)\mapsto \sup ((X_{r,u}\setminus\Int(X_{r,u})\cap (-\infty,t])$, that is,
$g(r,u,t)$ is the left endpoint of the complementary-in-$\I$ interval of $X_{r,u}\setminus\Int(X_{r,u})$ that contains $t$.
Then $g$ is definable and, for each $(r,u)$, the image $g(r,u,M_{r,u})$ is dense in $X_{r,u}\setminus\Int(X_{r,u})$.
Let $h\colon A\to \R$ send
$r+(\sigma(r)-r)t$ in $A$ to $g(\alpha(r), \beta(r),t)$;
then $h$ is definable and $h(A)$ is dense in $\I$, contradicting~(2).
(This finishes the case that $E$ is bounded.)

If $E$ is not bounded, apply the result to the image of $E$ under the map
$$
x\mapsto \left(x_1(1+x_1^2)^{-1/2},\dots,x_n(1+x_n^2)^{-1/2}\right)
$$
and then pull back.

For (3)$\Rightarrow$(4), set $m=0$.

(4)$\Rightarrow$(5).
Suppose that every $\Dsig$ has interior or is nowhere dense.
Let $E$ be $\Dsig$;
we must show that $\dim E=\tdim E$.
By countable stability, it suffices to consider the case that $E$ is compact.
Since $\tdim\leq \dim$ on compact sets, it suffices to show that $\tdim E\geq \dim E$.
By permuting coordinates, we reduce to the case that the projection of $E$ on the first $\dim E$ coordinates contains a box $V$.
Define $f\colon V\to \R^{n-\dim E}$ by letting $f(v)$ be the lexicographic minimum of $E_v$.
By basic real analysis (or \cf~\cite{dms1}*{2.8}), $\mathcal D(f)$ has no interior; by~\ref{D(f)} and~(4), it is nowhere dense.
Thus, there is a box $U\subseteq V$ such that $f\rest U$ is continuous, and so $u\mapsto (u,f(u))$ maps $U$ homeomorphically onto the graph of $f\rest U$.
Hence, $\tdim E\geq \tdim (f\rest U)=\dim E$.

For (6)$\Rightarrow$(1), recall~\ref{rotatecantor} and that $(\rbar,\N)$ defines all Borel sets.
%
\end{proof}



%

\section{Prelude to proof of Theorem~A}\label{S:fhmgen}

We establish in this section some basic cases of Theorem~A whose proofs use at most two more properties of $\adim$ in addition to those listed in~\ref{adimfactsintro}.
We do this for two reasons: (i)~certain key ideas are well illustrated in this simple setting; (ii)~there may be other notions of metric dimension beyond $\adim$ amenable to the same proofs.
Recall that by~\ref{uniform} we have already disposed of the $\tdim=\dim$ part of Theorem~A.

\begin{blank}\label{ominantifrac}
If $\rr^\circ$ is o\nbd minimal\tx, then $\dim=\adim$ on $\Dsig$.
\end{blank}

\begin{proof}
Since every $\Dsig$ is definable in $\rr^\circ$, it suffices to show that if $\rr$ is o\nbd minimal and $E$ is definable, then $\adim E=\dim E$.
By Kurdyka and Parusi\'{n}ski~\cite{kp} (or see Fischer~\cite{lstrat}), there exist
$N\in\N$ and bi-Lipshitz maps
$f_i\colon U_i\to \R^n$ for $i=1,\dots,N$ such that $E$ is the union of the $f_i(U_i)$, each $U_i$ is open in $\R^{n(i)}$, and $n_1\leq \dots\leq n_N=\dim E$.
By Lipschitz invariance, $\adim f_i(U_i)=\adim U_i=n(i)$; by
stability,
\begin{equation*}
\adim E=\max\set{n_i: i=1,\dots,N}=n_N=\dim E.\qedhere
\end{equation*}
\end{proof}

\subsubsection*{Remark}
The proof suggests that if $\rr$ is o\nbd minimal then
any reasonable notion of metric dimension would always coincide with $\dim$ (hence also $\tdim$) on all definable sets.

Next we dispose of the $\dim 0$ case.

\begin{blank}\label{dimzerocase}
If $E\in \Dsig$ and $0=\dim E<\Dim E$\tx, then $\rr$ defines $\N$. \end{blank}

We require for the proof two more properties of $\adim$ in addition to the first five listed in~\ref{adimfactsintro}:

\begin{dblank}\label{theconstant}
There exists $s>1$ such that $\adim A^2\geq s\adim A$ for all definable $A\subseteq\R$.
\end{dblank}

\begin{dblank}\label{zeropowers}
If $A\subseteq \R$ is definable and $\adim A=0$\tx, then $\adim A^k=0$.
\end{dblank}

(Indeed, $\adim E^k=k\adim E$ always, but we do not yet need this.)


\begin{blank}\label{falctrick}
If $A\subseteq \R$ is definable and $\adim A>0$\tx, then there exist $N\in\N$ and $F\colon \R^N\to\R$ definable in $\rbar$ such that $F(A^N)$ is dense.
\end{blank}


(Definability of $A$ is superfluous if one is interested only in Assouad dimension; see Remark~(ii) below after the proof.)

\begin{proof}
Let $Q\colon \R^4\to\R$ be given by
$$
Q(x)=
\begin{cases}
(x_1-x_2)/(x_3-x_4),& x_3\neq x_4\\
0,& x_3=x_4
\end{cases}
$$
We are done if $Q(A^4)$ is dense, so assume otherwise.
Then the vector-difference set of $A^2$ is disjoint
from some
open double cone $C\subseteq \R^2$ centered at the origin.
Let $\ell$ be the line through the origin perpendicular to the axis of~$C$ and let $\pi$ be the orthogonal projection of $\R^2$ onto $\ell$.
The restriction of $\pi$ to $A^2$ is bi-Lipschitz,
and so $\adim \pi(A^2)=\adim A^2\geq s\adim A$, where $s$ is as in~\ref{theconstant}.
Let $\theta$ be the rotation of the plane taking $\ell$ to the real line; then
$\adim (\theta\circ \pi(A^2))\geq s\adim A.
$
We now have a linear function $T_1\colon \R^2\to\R$ such that $\adim T_1(A^2)\geq s\adim A$.
We are done if $Q([T_1(A^2)]^4)$ is dense, so assume otherwise, and repeat the process above to obtain a linear function $T_2\colon \R^2\to\R$ that maps $[T_1(A^2)]^2$ to a set having $\adim$ at least $s^2\adim A$.
Since $0<\adim A\leq 1<s$, we obtain after finitely many repetitions
of this procedure some $m\in \N$ and linear $T\colon \R^m\to\R$ such that $Q(T(A^m)^4)$ is dense, thus producing $F$ as desired.
\end{proof}

\subsubsection*{Remarks}
(i)~Up to the choice of the various finitely many cones $C$ (equivalently, the rotations $\theta$) the construction of $F$ depends only on the least $k\in \N$ such that $s^k\adim A>1$.
(ii)~As mentioned earlier, $\adim E^k=k\adim E$ always. Hence, we could apply~\ref{falctrick} to $\rr=(\rbar,A)$ with $s=2$, and so the assumption of definability is not essential for the conclusion if one is interested only in Assouad dimension.

\begin{proof}[Proof of~\ref{dimzerocase}]
Let $E\in\Dsig$ and $0=\dim E<\adim E$. We must show that $\rr$ defines $\N$.

First, suppose that $E\subseteq \R$.
If $E$ is somewhere dense, then $\N$ is definable by~\ref{uniform}, so assume that $E$ is nowhere dense. Then $\dimcl E=0$, so it suffices to consider the case that $E$ is closed (since $\adim=\adimcl$).
Let $D$ be the set of midpoints of the bounded complementary intervals of $E$, together with $1+\max E$ if $E$ is bounded above.
Then $D$ is discrete and the function $g\colon D\to \R$ mapping
$t\in D$ to the left endpoint of the complementary interval of $E$ that contains $t$ is definable.
As $E$ is closed and has no interior, $g(D)$ is dense in $E$, so $\adim g(D) >0$ (because $\adimcl =\adim$).
With $F$ as in~\ref{falctrick} applied to $A=g(D)$, we have that $F(g(D)^N)$ is somewhere dense.
Hence, there exists $m\in\N$ and definable $h\colon D^m\to \R$ such that $h(D^m)$ is somewhere dense.
By~\ref{uniform}, $\rr$ defines $\N$.

We now dispose of the general case.
For $i=1,\dots,n$, let $\pi_iE$ be the projection of $E$ on the $i$\nbd th coordinate; then $\pi_i E\in\Dsig(1)$ and $\dim \pi_i E=0$.
As $E\subseteq (\bigcup_{i=1}^n \pi_iE)^n$,
we have
$\adim E\leq n\max(\adim \pi_iE)
$
by~\ref{zeropowers} and stability.
Hence, $\adim \pi_i E>0$ for some $i$, and so $\rr$ defines $\N$ by the previous paragraph.
\end{proof}

As an immediate corollary of~\ref{dimzerocase},

\begin{blank}\label{dimzerocasecor}
The following are equivalent.
\begin{itemize}
\item
$\rr$ does not define $\N$.
\item
$\dim=\adim$ on all $\dim 0$ $\Dsig$ sets.
\item
$\dim =\Dim$ on all $\Dsig$ sets in $\R$.
\end{itemize}
\end{blank}
\noindent
(Recall that $\adim\set{1/k: 0\neq k\in\N}\neq 0$.)

This is about as far as we can get by using only easily-stated basic facts like~\ref{theconstant}, \ref{zeropowers} and those of~\ref{adimfactsintro}, but a bit more discussion will help motivate some of the technicalities coming in the next section.
As we shall see, $\adim(A\times B)\leq \adim A+\adim B$ for all $A\in\R^j$ and $B\in\R^k$.
Hence, if the projection $\pi E$ of $E$ on the first $m$ coordinates has interior and the projection $\tau E$ of $E$ on the last $n-m$ coordinates has $\adim 0$, then $\adim E=\dim E$.
But the case $\adim\tau E=0$ will be exceptional, so what can be said in general?
Assume that $E$ is $\Dsig$ and $\rr$ does not define $\N$.
It is easy to see via the Baire Category Theorem
that the set $\set{x\in \R^m: \dim E_x>0}$ is meager; as it also $\Dsig$ (\ref{fiberdim}), it is nowhere dense by~\ref{uniform}.
Then $\set{x\in \R^m: \dim E_x=0}$ has interior, so by monotonicity, stability and~\ref{dimzerocase}, we may reduce to the case that $\adim E_x\leq 0$ for all $x\in \R^m$.
One can now imagine that under appropriate inductive assumptions we should be able to reduce to the case that $n=m+1$, and even that for every permutation $\sigma$ of coordinates, $\pi\sigma E$ has interior and
$\adim ((\sigma E)_x)\leq 0$ for every $x\in\R^m$.
But here our naive axiomatic approach seems to stall, even for $n=2$.
(The reader who wishes to get an idea of the difficulties could consider trying to show now that if $g\colon [0,1]\to \R$ is continuous and $(\rbar,g)$ does not define $\N$, then $\adim g=1$.)

\section{$\adim$, $\udim$, and the proof of Theorem~A}\label{S:dims}

Given a metric space $(X,\mathrm{d})$, $S\subseteq X$ and $r>0$
put
$$
\net_r S=\sup_{k\in\N}\exists x_1\dots x_k\in S,\ \bigwedge_{i\neq j}\mathrm{d}(x_i,x_j)\geq r.
$$
The \textbf{Assouad dimension} of $S$, denoted in this paper by $\adim S$, is the infimum of $\alpha\in\R$ such that
$$
\Set{(r/R)^{\alpha}\net_r\set{y\in S:\mathrm{d}(y,x)\leq R}:x\in S,\ 0<r<R<+\infty}
$$
is bounded.
Our primary source for information on $\adim$ is~\cite{luu}, but there: (a)~$\adim$ is denoted by $\dim_{\mathrm A}$,
$\tdim$ by $\dim$, and our notion of $\dim$ is not used in any explicit sense; (b)~the infimum is taken over $\alpha\geq 0$ (and so $\dim_{\mathrm A}\emptyset=0$, as opposed to $\adim\emptyset=-\infty$ by our definition).
In addition to the properties listed in~\ref{adimfactsintro}, $\adim$ is
sublogarithmic, and even logarithmic on cartesian \emph{powers}.

Recall that by~\ref{uniform} our goal is to show that $\adim=\dim$ on $\Dsig$ if $\rr$ does not define $\N$.
It turns out to be convenient to work with another notion of dimension for collections of sets.
We define the \textbf{upper Minkowski dimension} of $\mathcal S\subseteq \mathcal P(X)$ by
$$
\udim \mathcal S=\varlimsup_{\substack{r\downarrow 0\\
S\in\mathcal S}}\frac{\log\net_r S}{-\log r},
$$
(with $\log(0)=-\infty$ and $\log(+\infty)=+\infty$),
that is,
$\udim \mathcal S$ is the supremum of all $\alpha\in [-\infty,+\infty]$ for which there is a sequence $(r_k)$ of positive real numbers and a sequence $(S_k)$ of elements of $\mathcal S$ such that $\lim_{k\to +\infty} r_k=0$ and
$\lim_{k\to +\infty}(\log\net_{r_k}S_k)/(-\log r_k)=\alpha$.
For $S\subseteq X$, we write $\udim S$ instead of $\udim\{S\}$.
There are many different names for, and equivalent formulations of, $\udim$ on sets in the literature; it is often defined only for nonempty totally bounded sets $S$.
(In~\cite{luu}, $\udim$ on sets is denoted by $\overline{\dim}_{\mathrm B}$ and is called the upper box-counting box dimension.)

\subsubsection*{Note}
Equivalent metrics yield the same $\udim$ (and $\adim$).
Indeed, if
$f\colon (0,\delta)\to \R$ and $0<a<b\in \R$ are such that
$af(r)\leq \net_r S\leq bf(r)$ for all $r\in (0,\delta)$, then
$$
\udim S=\varlimsup_{r\downarrow 0}(\log f(r))/(-\log r).
$$
These observations allow for switching between alternate definitions of $\udim$ (and $\adim$) as convenient; see~\cite{falconerbook}*{3.1} for more details.


We shall obtain Theorem~A from the following more technical result.

\begin{thmb}
If $\rr$ does not define $\N$ and
$E$ is $\Dsig$ and bounded\tx,
then
$$
\udim \set{E_x: x\in \R^m \And \dim E_x=d}\leq d,\quad d=0,\dots,n-m.
$$
\end{thmb}

%
%
%
%
We begin to work toward the proof by listing some basic properties of $\udim$ in arbitrary metric spaces (proofs are easy and are left to the reader).

\begin{blank}\label{familyinequality}
$
\sup\set{\udim S: S\in\mathcal S}\leq \udim \mathcal S\leq \udim\bigcup\mathcal S
$.
\end{blank}



Recalling~\ref{familydim} for definitions, $\udim$ is increasing, stable, sublogarithmic, and logarithmic on cartesian \emph{powers}.
It is also closure invariant:

\begin{blank}$\udim\Cl\mathcal S=\udim\mathcal S$\tx, where
$\Cl\mathcal S:=\set{\Cl S: S\in \mathcal S}$.
\end{blank}

\begin{blank}
If $\mathcal A\subseteq \mathcal P(X^n)$\tx, then
$
\dim\mathcal A\leq \udim \mathcal A
$.
\end{blank}

%

%

If $Y$ is another metric space, then $\udim$ is
``Lipschitz nonexpansive'':

\begin{blank}
If $(f_S:S\to Y)_{S\in\mathcal S}$ is uniformly Lipschitz\tx, then $\udim\set{f_S(S): S\in\mathcal S}\leq \udim \mathcal S.$
\end{blank}

Hence, $\udim$ is also \textbf{Lipschitz invariant}:

\begin{blank}
If $(f_S:S\to Y)_{S\in\mathcal S}$ is uniformly bi-Lipschitz\tx, then $\udim\set{f_S(S): S\in\mathcal S}=\udim \mathcal S.$
\end{blank}


%

For $X=\R^n$, we have a special property.

\begin{blank}\label{ntrivial}
$\udim\mathcal A\leq n$ for every $\mathcal A\subseteq \mathcal P(\R^n)$ such that $\sup\set{\diam A:A\in\mathcal A}<+\infty$.
\end{blank}

When combined with~\ref{familyinequality},

\begin{blank}
If $E$ is bounded and some $E_x$ has interior \tx($x\in\R^m$\tx)\tx, then
$
\udim \set{E_x: x\in \R^m}=n-m.
$
\end{blank}

Hence,

\begin{blank}\label{dnminusm}
Theorem~B holds for $d=n-m$.
\end{blank}

Next is a uniform version of~\ref{dimzerocase}.

\begin{blank}\label{dzero}
Theorem~B holds for $d=0$.
\end{blank}

\begin{proof}
Assume that $\rr$ does not define $\N$ and $E$ is $\Dsig$ and bounded.
Let $A=\set{x\in \R^m: \dim E_x =0}$.
We must show that $\udim\set{E_x: x\in A}\leq 0$.

First, suppose that $n=m+1$.
By following the argument for~\ref{dimzerocase}, it suffices to exhibit a linear $T\colon \R^2\to\R$ such that
$
\udim\set{T((E_x)^2):x\in A}\geq 2\udim\set{E_x:x\in A}.
$
With $Q$ as in the proof of~\ref{falctrick}, observe that
$
\set{(x,t): x\in\R^{m}\And t\in Q((E_x)^4)}
$ is
$\Dsig$, and for every $x\in\R^m$,  $E_x$ has interior if and only if $Q((E_x)^4)$ has interior.
Hence, by~\ref{uniform}, there is a finite set $\mathcal I$ of intervals such that for every $x\in A$ there exists $I_x\in\mathcal I$ such that $Q((E_x)^4)$ is disjoint from $I_x$.
By arguing as in~\ref{falctrick} there is a finite collection
$\mathcal L$ of lines through the origin and positive real numbers $(c_\ell)_{\ell\in\mathcal L}$ such that, for each $x\in A$, there exists $\ell(x)\in \mathcal L$ such that the restriction  to $(E_x)^2$ of the orthogonal projection $\pi_{\ell(x)}$ onto $\ell(x)$ is bi-Lipshitz with lower constant $c_{\ell(x)}$.
For each such $\ell$, let $A_\ell$ be the set of $x\in A$ such that
$\pi_\ell\rest (E_x)^2$ is bi-Lipschitz with lower constant $c_\ell$.
Let $\theta_\ell$ be a rotation taking $\ell$ to the real line and put $T_\ell=\theta_\ell\circ \pi_\ell$.
As the family
$
(T_\ell\rest ((E_x)^2))_{x\in A_\ell}
$
is uniformly bi-Lipschitz, we have
$$
\udim\set{T_\ell((E_x)^2): x\in A_\ell}=\udim\set{(E_x)^2: x\in A_\ell}.
$$
By stability, for some $\ell\in\mathcal L$ we have
$$
\udim\set{(E_x)^2: x\in A}=\udim\set{(E_x)^2: x\in A_\ell}.
$$
Hence, for this $\ell$, we have
$$
\udim\set{T_\ell((E_x)^2): x\in A}\geq\udim\set{(E_x)^2: x\in A}.
$$
As $\udim$ is logarithmic on cartesian powers, we have
$$
\udim\set{T_\ell((E_x)^2): x\in A}\geq 2\udim\set{E_x: x\in \R^m}
$$
as desired.

With the case $n=m+1$ in hand, the general case follows by the uniform version of the last paragraph of the proof of~\ref{dimzerocase}.
\end{proof}

Hence,

\begin{blank}\label{dimzerothmb}
Theorem~B holds if either $n-m=1$ or
$\dim\set{E_x: x\in \R^m }\leq 0$.
\end{blank}

Again we have made it this far in axiomatic fashion, that is, via  easily-stated properties of $\udim$ rather than its definition; this now changes.
We shall finish the proof of Theorem~B by an induction, but we need to make the inductive assumptions rather more precise.
We require some notation and definitions (some of which are formulated specifically for this induction).

Let $\Pi(k,j)$ denote the collection of all coordinate projections $\R^k\to\R^j$, and $\Pi(k)=\bigcup_j\Pi(k,j)$.
Given $\pi\in\Pi(k,j)$ we let $\piperp\in \Pi(k,k-j)$ denote the orthogonal projection.
Observe that if $\pi\in \Pi(n,m)$ is the projection on the first $m$ coordinates and $x\in\R^m$, then $E_x=\piperp(E\cap \pi\inv(x))$.

Given a set $A$, its cardinality is denoted by $\card A$.

Put $B_k(E)=\set{v\in \Z^n: (v+\I^n)\cap 2^kE\neq \emptyset}$ and $N_k(E)=\card B_k(E)$.
Note that $N_k(E)$ is the number (allowing $+\infty$) of closed dyadic cubes in $\R^n$ of side length $2^{-k}$ that intersect~$E$, and that if $E\subseteq \I^n$, then $B_k(E)\subseteq \N^n$.
If $\mathcal S\subseteq \mathcal P(\R^n)$, then $B_k(\bigcup\mathcal S)=\bigcup_{S\in\mathcal S}B_k(S)$.
It is a standard fact that $\udim E=\varlimsup_{k\to+\infty}\log_2 N_k(E)/k$.
Hence,

\begin{blank}\label{boxcount}
If $\mathcal S\subseteq \mathcal P(\R^n)$\tx, then
$\udim\mathcal S\leq \alpha$ iff
$$
\forall \e>0 \exists k_\e\in\N \forall k\geq k_\e \forall S\in \mathcal S,\ N_k(S)\leq 2^{k(\alpha+\e)}.
$$
\end{blank}

The next batch of definitions is for tracking certain data that we need to make our induction go through.

For $\delta\in (0,\infty)$, we say that
$C\subseteq \N^n$ is \textbf{$\boldsymbol{\delta}$\nbd sparse} if there exists $\pi \in \Pi(n,n-1)$ such that $\card (\piperp(C\cap \pi\inv(u)))\leq \delta$ for all $u \in \N^{n-1}$.
(For $n=1$, this means just that $\card C\leq \delta$.)
If $\pi$ is the projection on the first $n-1$ coordinates, then this means that $\card(C_u)\leq \delta$ for all $u\in\N^{n-1}$.
An easy observation:

\begin{blank}\label{deltasparsefacts}
If $C\subseteq \N^n$ is $\delta$\nbd sparse\tx,
then there exists $\pi \in \Pi(n,n-1)$ such that
$\card C\leq \delta\card\pi C$.
%
%
\end{blank}

We need a uniform version for collections of sets. We say that
$\mathcal C\subseteq \mathcal P(\N^n)$ is \textbf{$\boldsymbol{\delta}$\nbd sparse} if there exists $\pi \in \Pi(n,n-1)$ such that for all $u \in \N^{n-1}$ and all $C\in\mathcal C$, we have $\card (\piperp(C\cap\pi\inv(u)))\leq \delta$.
The quantifier complexity of the next step is sufficiently complicated that we express it in logical notation.
We say that
$\mathcal S\subseteq \mathcal P(\I^n)$ is \textbf{sparse} if:
$$
\exists s\in \N\forall \e> 0 \exists k_\e \in \N \forall k>k_\e\exists
(C_{1,S})_{S\in\mathcal S}\dots (C_{s,S})_{S\in\mathcal S}
\forall S\in\mathcal S,\ B_k(S) = \bigcup_{i=1}^{s} C_{i,S}
$$
where each family $(C_{i,S})_{S\in\mathcal S}\subseteq \mathcal P(\N^n)$ is $2^{k\e}$\nbd sparse.
(We have suppressed any notation here for the dependence of the families $(C_{i,S})_{S\in\mathcal S}$ on $k$; we introduce it later as needed.)
We say that $s$ \textbf{witnesses} (or is a witness to) that $\mathcal S$ is sparse, and that $E$ is sparse if $E\subseteq \I^n$ and $\{E\}$ is
sparse.
Some easy observations (proofs are left to the reader):



\begin{blank}\label{sparsefacts}
Let $\mathcal S\subseteq \mathcal P(\I^n)$.
\begin{enumerate}
\item
If $\udim\mathcal S=0$\tx, then $\mathcal S$ is sparse and witnessed by $s=1$.
\item
If $\mathcal S$ is sparse with witness $s$ and $\mathcal S'\subseteq \mathcal P(\I^n)$ is sparse with witness $s'$\tx, then
$
\set{S\cup S': (S,S')\in \mathcal S\times \mathcal S' }
$
is sparse with witness $s+s'$.
\item\label{projectsparse}
If there exists $\pi\in \Pi(n)$
such that $\pi\mathcal S:=\set{\pi S: S\in\mathcal S}$ is sparse\tx,
then $\mathcal S$ is sparse with the same witness.
\item
If there exists $\pi\in \Pi(n)$
such that $\udim\pi\mathcal S=0$\tx,
then $\mathcal S$ is sparse and witnessed by $s=1$.
\end{enumerate}
\end{blank}

\begin{blank}\label{uniondefsparse}
If $E$ is the union of an increasing sequence $(S_r)_{r>0}$ of subsets of $\I^n$ and
$$
\set{S_{r,x}: r>0,\ x\in \I^m}
$$
is sparse\tx, then $\set{E_x: x\in \I^m}$ is sparse with the same witness.
\end{blank}


\begin{proof}
Let $s\in \N$ witness that $\set{S_{r,x}: r>0,\ x\in \I^m}$ is sparse.
Let $\e>0$.
We exhibit $k_\e$ such that for all $k>k_\e$ there are
$2^{k\e}$\nbd sparse families
$
(D_{k,1,x})_{x\in\I_m},\dots,(D_{k,s,x})_{x\in \I^m}
$
of subsets of $\N^{n-m}$
such that
$B_k(E_x)=\bigcup_{i=1}^s D_{k,i,x}$
for every $x\in \I^m$.
By assumption, there exists $k_\e \in \N$ such that for all $k>k_\e$ there are $2^{k\e}$\nbd sparse families
$$
(C_{k,1,r,x})_{(r,x)\in \R^{>0}\times \I^m},\dots,(C_{k,s,r,x})_{(r,x)\in \R^{>0}\times \I^m}
$$
of subsets of $\N^{n-m}$ such that
$
B_k(S_{r,x}) = \bigcup_{i=1}^s C_{k,i,r,x}
$
for every $(r,x)\in \R^{>0}\times \I^m$.
Let $k>k_\e$ and $x\in \I^m$.
Since $E_x\subseteq \I^{n-m}$, we have
$B_k(E_x)\subseteq \{0,\dots,2^k\}^{n-m}$.
Since $(S_r)$ is increasing in $r$, so is $(S_{r,x})$.
Thus, there exists $r(k,x)>0$ such that
$
B_k(E_x)=B_k(S_{r(k,x),x})
$.
Hence,
$$
B_k(E_x)=\bigcup_{i=1}^sC_{k,i,r(k,x),x}.
$$
Since each $(C_{k,i,r,x})_{(r,x)}$ is $2^{k\e}$\nbd sparse, so is each $(C_{k,i,r(k,x),x})_{x\in \I^m}$.
Put $D_{k,i,x}=C_{k,i,r(k,x),x}$.
Then
$
(D_{k,1,x})_{x\in\I^m},\dots,(D_{k,s,x})_{x\in \I^m}
$
are as desired.
\end{proof}

\subsection*{Proof of Theorem~B}
Let $E$ be $\Dsig$ and bounded.
We must show that
$$
\udim\set{E_x: x\in\R^m\And \dim E_x=d}\leq d.
$$
We have already established the cases $n-m=1$, $n-m=d$, and $d=0$.
By translation, dilation and Lipschitz invariance, we may take $E\subseteq \I^n$.
We now proceed by induction on $n-m\geq 2$ to establish
the following:
\begin{itemize}
\item[]
(i)$_{n,m}$
$\set{E_x: x\in \I^m\And \Int(E_x)=\emptyset}$ is sparse.
\item[]
(ii)$_{n,m}$
$\udim\set{E_x: x\in\I^m\And \dim E_x=d}\leq d$,\quad
$1\leq d\leq n-m-1$.
\end{itemize}

\subsubsection*{Remark}
Formally, we could start the induction with $n-m=1$, but
the key ideas will be better illustrated by doing the case $n-m=2$ explicitly.

\subsubsection*{Proof of \tx{(i)}$_{m+2,m}$}

Suppose that $n=m+2$.
Let $A=\set{x\in \I^m: \Int(E_x)=\emptyset}$.
We must show that $\set{E_x: x\in A}$ is sparse.
Write $E=\bigcup_{r>0}X_r$ as in~\ref{Dsigequiv}.
By~\ref{uniondefsparse}, it
suffices to show that the collection
$
\set{X_{r,x}:r>0,\ x\in A}
$
is sparse.
In order to improve readability, we first suppress the parameters and show that \tx{(i)}$_{2,0}$ holds, that is, if $E\subseteq \I^2$ ($\cong \I\times \I$) is definable, compact and has no interior, then $E$ is sparse.
The argument will be explicit enough to read off the needed uniformity in parameters.

Define $f\colon \I^2\to \R$ by
$f(x,y)=\min (E_x\cup \{2\})\cap [y,\infty)$.
For $y \in \I$, let $f(\plc,y)$ denote the function $x\to f(x,y)\colon  \I\to \R$.
Put
$
F=E\cup \set{ (x,y) \in \I^2: x \in \mathcal D(f(\plc,y))}.
$
For $y\in \R$, put $E^y=\set{x\in \R: (x,y)\in E}$ and $F^y=\set{x\in \R: (x,y)\in F}$.
For each $k$, put
\begin{align*}
C_{k,1}&=B_k(\set{(x,y)\in\I^2: \Int(E_x)\neq \emptyset})\\
C_{k,2}&=B_k(\set{(x,y)\in\I^2: \Int(E^y)\neq \emptyset})\\
C_{k,3}&=\set{(u,v)\in B_k(E): u\in B_k(F^{v/2^k})}\setminus C_{k,2}\\
C_{k,4}&=B_k(E)\setminus (C_{k,1}\cup C_{k,2}\cup C_{k,3})
\end{align*}
As $B_k(E)$ is the union of the $C_{k,i}$, it suffices to
let $\e>0$ and show that, for each $i\in \{1,2,3,4\}$,
$C_{k,i}$ is $2^{k\e}$\nbd sparse for all sufficiently large $k$.

We first dispose of the $C_{k,1}$ and $C_{k,2}$.
As $E$ is $\fsig$ and has no interior, the same is true of
$
\set{x\in \I: \Int (E_x)\neq \emptyset}
$; as this set is $\Dsig$ (\ref{fiberdim}) and contained in $\R$, it has  $\udim$ at most $0$ by~\ref{dimzerothmb}.
Thus, for sufficiently large $k$, the projection of $C_{k,1}$ on the first coordinate has cardinality at most $2^{k\e}$; for such $k$, the set $C_{k,1}$ is $2^{k\e}$\nbd sparse (by definition).
The $C_{k,2}$ are similarly handled via projection on the second coordinate.

\begin{figure}[b]
\begin{tikzpicture}[scale=0.75, transform shape,
    thick,
    >=stealth',
    dot/.style = {
      draw,
      fill = white,
      circle,
      inner sep = 0pt,
      minimum size = 4pt
    }
]

\definecolor{red}{gray}{0.1}
\definecolor{green}{gray}{0.3}
\definecolor{blue}{gray}{0.4}
\definecolor{yellow}{gray}{0.6}
\definecolor{violet}{gray}{0.5}
\definecolor{orange}{gray}{0.3}

    \draw[very thin,color=gray] (0,0) grid (6.1,6.1);
    \draw[->] (0,0) -- (6.2,0) node[right] {$x$};
    \draw[->] (0,0) -- (0,6.2) node[above] {$y$};
   \draw[green, line width=2pt]      (1.3,0) -- (1.3,6);

  \draw[red] plot[smooth] coordinates {(0,2.5) (0.12,2.6) (0.6,2.4) (1,2.8)  (1.4,2.6) (1.7,2.8) (2,2.75) (2.1,2.2) (2.2,1.6) (2.31,1)(2.4,.85) (2.8,.87) (3,.86) (3.25,.80) (3.5,.92)};
  \draw[red] plot[smooth] coordinates {(0.4,0) (0.5,0.15) (0.6,0.2) (0.8,0.4) (0.9,0.38) (1,.5)  (1.12,.7) (1.5,.6) (2,.65) (2.2,0.2) (2.3,0)};
  \draw[red] plot[smooth] coordinates {(0.6,4.5) (0.8,4.35) (0.9,4.4) (1,4.2)(1.12,4.05) (1.6,4.1) (2,4.3) (2.3,4.2) (2.8,4.5) (3,4.7) (3.1,4.65) (3.2,4.4) (3.3,4.2) (3.5,4) (3.6,3.9) (3.8,3.85) (4,3.88) (4.1, 3.9) (4.2,4.2) (4.3,4.6) (4.4, 4.8) (4.5,5) (4.6,5.3) (4.7,5.8) (4.8,6)};
  \draw[red] plot[smooth] coordinates {(1.45,6) (1.5,5.4) (1.6,5.3) (1.8,5.5) (2,5.4) (2.31,5.35)};
  \draw[red] plot[smooth] coordinates {(2.31,5) (2.4,5.2) (2.8,5.87) (3,5.8)  (3.2,5.6) (3.5,5.85) (3.7,5.8) (4,5.6) (4.1,5.5) (4.3,5.45) (4.5,5.35)};
  \draw[red] plot[smooth] coordinates {(3.5,0) (3.6,.1) (3.7,.15) (4,.05) (4.2,.2) (4.4,.5) (4.6,.3) (4.8,.4) (5,.45) (5.2,.6) (5.4,.5) (5.6,.6) (5.8,.7) (6,.75)};
  \draw[red] plot[smooth] coordinates { (4.25,1) (4.3,1.2) (4.5,1.4) (4.7,1.3) (5,1.35) (5.3,1.3) (5.6,1.4) (5.7,1.2) (6,1.1)};

   \draw[thin,pattern=north west lines, pattern color=blue,opacity=.2]   (0,0) rectangle (1,1);
   \filldraw[thin,orange,opacity=.2](0,2) rectangle (1,3);
  \draw[thin,pattern=north west lines, pattern color=blue,opacity=.2] (0,4) rectangle (1,5);
      \filldraw[thin,orange,opacity=.2] (1,0) rectangle (2,1);
  \filldraw[thin,orange,opacity=.2] (1,2) rectangle (2,3);
     \filldraw[thin,orange,opacity=.2](1,4) rectangle (3,5);
  \draw[thin,pattern=north west lines, pattern color=blue,opacity=.2](1,5) rectangle (2,6);
  \draw[thin,pattern=north west lines, pattern color=blue,opacity=.2] (2,0) rectangle (3,1);
  \draw[thin,pattern=north west lines, pattern color=blue,opacity=.2](2,1) rectangle (3,3);
  \draw[thin,pattern=north west lines, pattern color=blue,opacity=.2] (2,5) rectangle (3,6);
  \filldraw[thin,orange,opacity=.2](3,5) rectangle (4,6);
 \draw[thin,pattern=north west lines, pattern color=blue,opacity=.2](3,4) rectangle (4,5);
 \draw[thin,pattern=north west lines, pattern color=blue,opacity=.2] (3,0) rectangle (4,1);
 \draw[thin,pattern=north west lines, pattern color=blue,opacity=.2] (3,3) rectangle (4,4);
  \filldraw[thin,orange,opacity=.2] (4,0) rectangle (5,1);
 \draw[thin,pattern=north west lines, pattern color=blue,opacity=.2] (4,3) rectangle (5,6);
 \draw[thin,pattern=north west lines, pattern color=blue,opacity=.2] (4,1) rectangle (5,2);
  \filldraw[thin,orange,opacity=.2] (5,1) rectangle (6,2);
  \filldraw[thin,orange,opacity=.2] (5,0) rectangle (6,1);
   \end{tikzpicture}
\caption{\ }
\end{figure}

(Referring now to Figure~1:
The wavy lines indicate graphs of the various $f(\plc,n/2^k)$; the striped boxes are those with bottom left corner in
$2^{-k}C_{k,3}$; and the shaded boxes are those with bottom left corner
in $2^{-k}C_{k,4}$.
Note that while each nonempty $E_x$ must intersect every shaded box, it need not intersect every striped box; this observation is key in the next stage of the proof.)

Now we deal with the $C_{k,3}$.
By compactness of $E\cup (\I\times\{2\})$, each $f(\plc,y)$ is lower semicontinuous.
It is routine real analysis that for each $\eta>0$ the set
$$E\cup \set{ (x,y) \in \I^2: x \in \mathcal D(f(\plc,y),\eta)}$$
is compact---and so $F$ is $\Dsig$ (recall the argument for~\ref{D(f)})---and for each $y\in \I$,
$F^y$ has interior if and only if $E^y$ has interior.
If $y\in \I$ and $E^y$ has no interior, then
$
F^y\setminus \Int(F^y)=E^y\cup \mathcal D(f(\plc,y))
$,
and so
$
\udim\set{F^y: y\in \I\And \Int(E^y)=\emptyset}=0
$
by applying~\ref{dimzerothmb} to $\set{(y,x): (x,y)\in F}$.
By~\ref{sparsefacts},
$
\set{F^y: y\in \I\And \Int(E^y)=\emptyset}
$
is sparse (with witness~$1$).
Hence, $C_{k,3}$ is $2^{k\e}$\nbd sparse for all sufficiently large $k$, because
$\card\set{u\in\N: (u,v)\in C_{k,3} }\leq \card N_k(F^{v/2^k})$
for each $v\in\N$.

Finally, we deal with the $C_{k,4}$.
By~\ref{dzero}, we have
$
\udim\set{E_x: \Int(E_x)=\emptyset}=0.
$
Hence, for all $x\in \I$ and sufficiently large $k$, if $E_x$ has no interior, then $N_k(E_x)\leq 2^{k\e}$;
we show that $C_{k,4}$ is $2^{k\e}$\nbd sparse for such $k$ by letting $(u,v)\in C_{k,4}$ and showing that $(C_{k,4})_u\subseteq B_k(E_{u/2^k})$.
As $C_{k,4}$ is disjoint from $C_{k,3}$, the function
$
x\mapsto f(x,v/2^k):[u/2^k,(u+1)/2^k]\to\R
$
is continuous.
Moreover, its range is contained in $[v/2^k,(v+1)/2^k]$, for if not, then $u\in B_k(E^{(v+1)/2^k})$ by the Intermediate Value Theorem, contradicting the definition of $C_{k,4}$.
Hence, $v\in B_{k}(E_{u/2^k})$ (as desired).

We have finished the proof of the case $m=0$ and $n=2$.
For arbitrary $m$, recall the compact sets $X_{r,x}$ at the beginning of the proof.
Put
$
f(r,x,a,b)=\min(X_{r,x,a}\cup\{2\})\cap [b,\infty)
$
for $(a,b)\in\I^2$.
Make the obvious modifications to the definitions of the $C_{i,k}$, and then proceed mutatis mutandis.
(This ends the proof of \tx{(i)}$_{m+2,m}$.)

\subsubsection*{Proof of \tx{(ii)}$_{m+2,m}$}
Suppose that $n=m+2$.
Put
$
A(1)=\set{x\in \R^m: \dim E_x=1}.
$
We must show that $\udim \set{E_x: x\in A(1)}\leq 1$.
Suppose to the contrary that 
there exist $\e>0$ and a sequence $(x_k)$ in $A(1)$ such that
\begin{equation}
\forall k,\ N_k(E_{x_k}) > 2^{k(1+\e)}.
\end{equation}
For each $\pi\in \Pi(2,1)$, we have
$\udim\set{\pi (E_x): x\in\R^m}\leq 1$.
Hence, after throwing away finitely many $x_k$, we reduce to the case that also
\begin{equation}
\forall k\forall \pi\in \Pi(2,1),\ N_k(\pi(E_{x_k}))<2^{k(1+\e/3)}.
\end{equation}
By $\text{(i)}_{m+2,m}$, $\set{E_x: x\in A(1)}$ is sparse.
Thus, there exist $s, j\in \N$ such that $B_j(E_{x_j})$ is a union of $s$\nbd many $2^{j\e/2}$\nbd sparse sets and $
2^{j(1+\e/2)}\geq s 2^{j(1+\e/3)}
$.
By~\ref{deltasparsefacts} and~(1), there exists $\pi \in \Pi(2,1)$ such that
$
N_j(\pi(E_{x_j})) \geq  N_j(E_{x_j})/s2^{j\e/2}>2^{j(1+\e/3)},
$
contradicting~(2).
(This ends the proof of \tx{(ii)}$_{m+2,m}$.)

\subsubsection*{Inductive assumption}
Let $n-m>2$.
Assume that
$\text{(i)}_{n',m'}$ and $\text{(ii)}_{n',m'}$
hold for all
$n',m'\in \N$ such that $0\leq n'-m'<n-m$.

\subsubsection*{Proof of \tx{(i)}$_{n,m}$}

This is quite similar to the proof of \tx{(i)}$_{m+2,m}$, so we only hint at the needed modifications.
As before, we reduce to the case that $m=0$, that is, we show that if $E$ is compact and has no interior, then $E$ is sparse (and explicitly so).
Define $f$ as before, but with $x$ ranging over $\I^{n-1}$, and define
%
%
the associated set $F$ in the obvious way.
Let $\pi$ denote projection on the first $n-1$ coordinates.
Consider the following sets and collections, where $E^y$ and $F^y$ are  defined similarly as before:
\begin{enumerate}
\item
$\set{(x,y)\in E: x\notin \Int\pi E}$
\item
$\set{(x,y)\in E: x\in \Int\pi E\And \Int(E_x)\neq \emptyset}$
\item
$\set{(x,y)\in E: \Int(E^y)\neq \emptyset}$
\item
$\set{E_x: x\in \I^{n-1}\And \Int(E_x)=\emptyset}$
\item
$\set{F^y: y\in \I\And \Int(E^y)=\emptyset}$
\end{enumerate}
Since $\pi E\setminus \Int\pi E$ is $\Dsig$ and has no interior, it is sparse by~\tx{(i)}$_{n-1,0}$; then set~(1) is also sparse, because
$
\pi\set{(x,y)\in E: x\notin \Int\pi E}=\pi E\setminus \Int\pi E.
$
Similarly, sets~(2) and~(3) are sparse because the sets
$$
\set{x\in \Int\pi E:\Int(E_x)\neq \emptyset}\quad
\set{y\in\I: \Int (E^y)\neq \emptyset}
$$
are $\Dsig$ and have no interior.
By~\tx{(i)}$_{n,n-1}$, collection~(4) is sparse.
By~\tx{(i)}$_{n,1}$ (applied to $\set{(y,x): (x,y)\in F}$),  collection~(5) is sparse.
Now proceed mutatis mutandis as for~\tx{(i)}$_{2,0}$.

\subsubsection*{Proof of \tx{(ii)}$_{n,m}$}
Let $d\in \{1,\dots,n-m-1\}$ and
$
A(d)=\set{x\in \R^m: \dim E_x=d}.
$
We must show that $\udim\set{E_x: x\in A(d)}\leq d$.
If $d=n-m-1$, then the argument is essentially the same as for \tx{(ii)}$_{m+2,m}$.
If $d<n-m-1$, then we apply \tx{(ii)}$_{n-1,m}$ to $\pi E$ for each $\pi\in \Pi(n-m,n-m-1)$, then finish as for \tx{(ii)}$_{m+2,m}$ via the obvious modifications.

This ends the proof of Theorem~B. \qed

\subsection*{Proof of Theorem A}
Let $E$ be $\Dsig$.
We must show that $\tdim E=\dim E=\adim E$.
By~\ref{uniform}, it suffices to show that $\adim E=\dim E$.
For $(x,t)\in \R^n\times \R$, put
$$
\phi(x,t)=(x_1(1+x_1^2)^{-1/2},\dots,
x_n(1+x_n^2)^{-1/2},t(1+t^2)^{-1/2}).
$$
The set
$$
F:=\set{(\phi(x,t),(y-x)/t): t>0;\ x,y\in E;\ \abs{x-y}\leq t}
$$
is $\Dsig$ and bounded.
For each $(x,t)\in E\times \R^{>0}$, we have
$$F_{\phi(x,t)}=[-1,1]^n\cap (E-\{x\})/t.$$
Thus, $\dim F_{\phi(x,t)}=\dim E$, so by Theorem~B,
$$
\dim E=\udim\set{[-1,1]^n\cap (E-\{x\})/t: (x,t)\in E\times \R^{>0}}.
$$
By translation invariance,
$$
\dim E=\udim\set{(E\cap \bar B(x,t))/t: (x,t)\in E\times \R^{>0}}
$$
where $\bar B(x,t)$ is the closed cube of side length $2t$ centered at $x$.
Let $\e>0$.
We show that $\adim E\leq \e+\dim E$.
There exists $\rho>0$ such that for all $x\in E$, $t>0$ and $0<s<\rho$, we have
$
\net_{s}(E\cap \bar B(x,t))/t)\leq s^{-(\e+\dim E)}.
$
Put $C=\max(1,\net_\rho \bar B(0,1))$.
Let $x\in E$ and $0<r<R$.
If $r/R<\rho$, then
$$
\net_r(E\cap \bar B(x,R))=\net_{r/R}((E\cap \bar B(x,R))/R)\leq (R/r)^{\e+\dim E}\leq C(R/r)^{\e+\dim E}.
$$
If $r/R\geq \rho$, then
\begin{multline*}
\net_r(E\cap \bar B(x,R))\leq \net_{\rho R}(E\cap \bar B(x,R))\\
\leq \net_{\rho R}\bar B(x,R)=\net_\rho \bar B(0,1)\leq C(R/r)^{\e+\dim E}.
\end{multline*}
Hence,
$\adim E\leq \e+\dim E$, as desired. \qed

\section{Concluding remarks}\label{S:conc}

We have already noted in the introduction that Theorem~A applies beyond the easy case that $\rr^\circ$ is o\nbd minimal.
Here is another class of examples not yet described in the literature.

\begin{blank}\label{infrank}
There is a Cantor set $E\subseteq \R$ such that\tx,
if $\rr$ is o\nbd minimal and exponentially bounded\tx, then
$\tdim=\adim=\dim$ on all sets
definable in the expansion of $\rr$ by all subsets of $E^m$ \tx($m$ ranging over $\N$\tx) that have countable closure.
\end{blank}

(See any of~\cites{geocat,fkms,tameness} for the definition of ``exponentially bounded''. Note that every Cantor subset of $\R$ contains countable closed sets of arbitrary countable Cantor-Bendixson rank.)

\begin{proof}
Let $E$ be as in the proof of~\cite{fkms}*{Theorem~B}.
By~\ref{intorcntbl}, it suffices to let $X\subseteq E$ be countable and closed, and show that every bounded unary set definable in $(\rr,X)^{\#}$ has interior or is countable, where $(\rr,X)^{\#}$ denotes the expansion of $\rr$ by all subsets of each finite cartesian power of $X$.
By~\cite{fkms}*{1.11}, it suffices to let $f\colon [0,1]^n \to \R$
be bounded and definable in $\rr$, and show that $f(X^n)$ has   countable closure.
By compactness of $X^n$, it suffices to let $x\in X^n$ and find $\delta>0$ such that $f(X^n \cap B(x,\delta))$ has countable closure.
The result now follows from~\cite{fkms}*{1.8}, assertion~(ii)
in the proof of~\cite{fkms}*{Theorem~B}, and the countability and compactness of $X^n$.
\end{proof}

\begin{dblank}
Consideration of $\set{k!: k\in\N}$ shows that the dimensional coincidence $\tdim=\dim=\adim$ can hold for sets
that define $\N$ over $\rbar$.
\end{dblank}


\begin{dblank}
In~\ref{fsigintro}, ``$\fsig$''  cannot be relaxed to ``boolean combination of $\fsig$'': By~\cite{densepairs}*{Theorem 1}, if $E$ is $\fsig$ and a real-closed proper subfield of $\R$ (say, if $E$ is the set of all real algebraic numbers),
then every set definable in $(\rbar,E)$ is a boolean combination of $\fsig$ sets. But both $E$ and $\R\setminus E$ are dense in~$\R$.
\end{dblank}

\begin{dblank}
The conclusion of Theorem~A does not generally extend to definable metric spaces (that is, metric spaces of the form $(E,d)$ where $d\colon E^2\to \R^{\geq 0}$ is definable).
For example, $(x,y)\mapsto \abs{x-y}^{2/3}$ is a metric on $\I$, and it is easy to see that $\hdim\I=3/2$ with respect to this metric.
\end{dblank}

\subsection*{Open issues and further directions}
While we regard Theorem~A as a major advance in the tameness program over $\rbar$, there are several questions left open.
Perhaps the most important of these: If $\rr$ does not define $\N$, does $\tdim=\adim=\dim$ on all sets definable in the open core of $\rr$? To put this another way, if $\rr$ is an expansion of $\rbar$ by constructible sets and it does not define $\N$, does $\tdim=\adim=\dim$ on all definable sets?
By~\ref{tdimtdimclthm}, this is the same as asking whether $\tdim=\tdimcl$ on all definable sets, but we do not yet know if $\dim=\dimcl$ on all definable sets (a weaker condition, on the face of it). Indeed, we do not yet know what to say about the dimensions of complements of $\Dsig$ sets
(but we are working on it).

Are there notions of dimension that satisfy all of the properties listed in~\ref{adimfactsintro} (at least, on compact sets) but are not bounded above by $\adim$?

What can be said about $\rr$ under the assumption $\tdim\leq \dim$ on all definable sets? As mentioned earlier, there are $\gdelta$ sets in $\R^3$ with $\tdim>\dim$ (\cite{engelking}*{1.10.23}), so $\N$ is not definable (and thus Theorem~A applies). But beyond this?

There is a notion dual to $\adim$ in the literature; see Fraser~\cite{fraser} for information.
We have not considered it here as it allows for nonempty open subsets of $\R$ to have dimension zero (and it is decreasing, not increasing) but perhaps there are still some associated tameness results.

We have omitted from this paper consideration of applications of Theorems~A and~B beyond those to the tameness program given in the introduction, in particular, we forego describing in detail obvious strengthenings of already-known results (see~\cite{tameness} for a few).
But there are plenty of open questions to be pursued.
Just one example: Suppose that $E$ is closed and definable and $p\in \N$.
Is there a definable $C^p$ function $f\colon \R^n\to\R$ such that $E=f\inv(0)$?
If $\rr$ defines $\N$, then yes, via a result of H.~Whitney (indeed, $f$ can then be taken to be $C^\infty$ and real-analytic off $E$).
If $\rr$ is o\nbd minimal, then again, yes (\cite{geocat}*{4.22}); see also Miller and Thamrongthanyalak~\cite{cpzero}.
But in general?
What if moreover $\dim=\dimcl$ on all definable sets?
Or if $\tdim=\tdimcl$ on all definable sets?
(And so on.)

In order to increase the accessibility of this paper, we have omitted any consideration of model-theoretic (as opposed to just definability-theoretic) tameness.
Model theorists will understand that there are numerous questions in this direction to pursue.

%
%
%
%

\bibsection

\end{document}